\documentclass[10pt,twoside,a4paper]{article}
\usepackage{euscript,a4,times}
 \usepackage[cp866]{inputenc}
  \usepackage[english]{babel}
   \usepackage{makeidx}
    \usepackage{latexsym,amsfonts,amssymb,amsmath,longtable,amsthm}
     \usepackage{tikz}
      \usetikzlibrary{arrows}
       \tikzstyle{block}=[draw opacity=0.7,line width=1.4cm]
\input amssym.def
\numberwithin{equation}{section}
\newcommand{\be}{\begin{equation}}
 \newcommand{\ee}{\end{equation}}
  \newcommand{\bea}{\begin{eqnarray}}
   \newcommand{\eea}{\end{eqnarray}}
\newcommand{\Rn}{\mathbb{R}^n}
 \newcommand{\Rnn}{\mathbb{R}^{n\times n}}
  \newcommand{\R}{\mathbb{R}}
   
    \newcommand{\N}{\mathbb{N}}
     \newcommand{\V}{\mathcal{V}}
      \newcommand{\G}{\mathcal{G}}
 \newcommand{\dist}{\mathrm{dist}}
  \newcommand{\Div}{\mathrm{div}}
   \newcommand{\sprt}{\mathrm{supp}}
    \newcommand{\dx}{\,\mathrm{d}x}

        \newcommand{\diam}{\mathrm{diam}}
         
 \newcommand{\Om}{\Omega}
  \newcommand{\om}{\omega}
   
    \newcommand{\intom}{\int_\Om}

 \newcommand{\test}{C_{\rm c}^\infty}
  \newcommand{\wtest}{W_{\rm c}^{1,p}}
   \newcommand{\wdtest}{W_0^{1,p}}
 \newcommand{\D}{\mathcal{D}_0^{1,p}}
  \newcommand{\dD}{\mathcal{D}^{-1,p'}}
   
    \newcommand{\W}{W_0^{1,p}}
     
      \newcommand{\Wloc}{W_{\rm loc}^{1,p}}

   \newcommand{\e}{\varepsilon}
    \newcommand{\Ln}{\mathcal{L}^n}
     
\newtheorem{theorem}{Theorem}[section]
 \newtheorem{proposition}[theorem]{Proposition}
  \newtheorem{lemma}[theorem]{Lemma}
   \newtheorem{corollary}[theorem]{Corollary}
    \theoremstyle{definition}
     \newtheorem{definition}[theorem]{Definition}
      
       \newtheorem{remark}[theorem]{Remark}
        
\newtheorem*{corollary*}{Corollary}
\newtheorem*{theorem*}{Theorem}
\newtheorem*{definition*}{Definition}

\title{\Large{{\bf On positive solutions of the $(p,A)$-Laplacian\\ with a potential in Morrey space}}}
\author{\small{YEHUDA PINCHOVER}
        \\
        \small{\&}
        \\
        \small{GEORGIOS PSARADAKIS}
        \\
        \date{\today}}
\begin{document}
\maketitle
\begin{abstract} We study qualitative positivity properties of quasilinear equations of the form
\[
 Q'_{A,p,V}[v] := -\Div(|\nabla v|_A^{p-2}A(x)\nabla v) + V(x)|v|^{p-2}v =0 \qquad x\in\Omega,
\]
where $\Omega$ is  a domain in  $\R^n$,  $1<p<\infty$, $A=(a_{ij})\in L^\infty_{\rm loc}(\Om;\Rnn)$ is a symmetric and locally uniformly positive definite matrix, $V$ is a real potential in a certain local Morrey space (depending on $p$), and
\[
 |\xi|_{A}^{2}:=A(x)\xi\cdot\xi=\sum_{i,j=1}^n a_{ij}(x)\xi_i\xi_j \qquad x\in\Om,~\xi=(\xi_1,\ldots,\xi_n)\in\R^n.
\]
Our assumptions on the coefficients of the operator for $p\geq 2$ are the minimal (in the Morrey scale) that ensure the validity of the local Harnack inequality and hence the H\"older continuity of the solutions. For some of the results of the paper we need slightly stronger assumptions when $p<2$.

We prove an Allegretto-Piepenbrink-type theorem for the operator $Q'_{A,p,V}$, and extend criticality theory to our setting. Moreover, we establish a Liouville-type theorem and obtain some perturbation results. Also, in the case $1<p\leq n,$ we examine the behavior of a positive solution near a nonremovable isolated singularity and characterize the existence of the positive minimal Green function for the operator $Q'_{A,p,V}[u]$ in $\Om$.

  \medskip

  \noindent\footnotesize{\textbf{2010 Mathematics Subject Classification}: Primary 35J92; Secondary 35B09, 35B50, 35B53, 35J08.}

  \medskip

  \noindent\footnotesize{\textbf{Keywords and phrases}: Quasilinear elliptic equation, Liouville theorem, maximum principle, minimal growth, Morrey spaces, $p$-Laplacian, positive solutions, removable singularity}
\end{abstract}
\tableofcontents
\section{Introduction}\label{sec:introduction}
Let $\Om$ be a domain in $\Rn$, $n\geq 2$. The Allegretto-Piepenbrink (AP) theorem asserts that under some regularity assumptions on a real symmetric matrix $A$ and a real potential $V$, the nonnegativity of the Dirichlet energy
 \be\nonumber
  \int_\Om\Big(|\nabla u|^2_A+V(x)|u|^2\Big)\dx\geq0
   \qquad \mbox{for all }u\in\test(\Om),
 \ee
is \textit{equivalent} to the existence of a positive weak solution of the Schr\"{o}dinger equation
 \be\label{intro:Schrodinger}
  -\Div\big(A(x)\nabla v\big)+V(x)v=0 \qquad \mbox{in }\Om,
 \ee
where
 \be
  |\xi|_{A}^{2}:=A(x)\xi\cdot\xi=\sum_{i,j=1}^n a_{ij}(x)\xi_i\xi_j\geq 0 \qquad \forall x\in \Omega, \mbox{ and }\forall \xi=(\xi_1,\ldots,\xi_n)\in\R^n.
 \ee

After the original results in \cite{All1}, \cite{P1}, a sequence of papers gradually relaxed the assumptions on $A$ and $V$  (see \cite{P2}, \cite{MP}, \cite{All2} and \cite{All3}). It was established by Agmon in \cite{Agm} that if $A\in L^\infty_{\rm loc}(\Om;\Rnn)$ is symmetric and locally uniformly positive definite in $\Om$, and $V\in L_{\rm loc}^q(\Om)$ with $q>n/2$, then the AP theorem holds true. If $A$ is the identity matrix, further relaxation on the regularity of $V$ is established in \cite[\S C8]{Sm}, albeit some global condition on $V^-$ is required there. We refer to \cite{LStV} and references therein for an up to date account.

A generalization of the AP theorem to certain quasilinear equations with $A$ being the identity matrix and $V\in L^\infty_{\rm loc}(\Om)$ has been carried out in \cite{PT1}. This was recently extended in \cite{PR} to include Agmon's assumptions on the matrix $A$. More precisely, for $1<p<\infty$, $A$  as above, and $V\in L_{\rm loc}^\infty(\Om)$, the nonnegativity of the functional
 \be\label{intro:natural form bound}
  Q_{A,p,V}[u]:=\int_\Om\Big(|\nabla u|^{p}_A+V(x)|u|^p\Big)\dx\geq0
   \qquad \mbox{for all }u\in\test(\Om),
 \ee
is proved to be equivalent to the existence of a positive weak solution to the corresponding Euler-Lagrange quasilinear equation
 \be\label{intro:pLaplace}
  Q'_{A,p,V}[u]:=-\Div\big(|\nabla v|^{p-2}_{A} A(x)\nabla v\big)+V(x)|v|^{p-2}v=0
   \qquad \mbox{in }\Om.
 \ee
Clearly, the quasilinear equation (\ref{intro:pLaplace}) satisfies the homogeneity property of equation (\ref{intro:Schrodinger}) but not the additivity (such an equation is sometimes called {\em half-linear}). Consequently, one expects that positive solutions of (\ref{intro:pLaplace}) would share some properties of positive solutions of (\ref{intro:Schrodinger}).

An essential common implication of the various assumptions on $A$ and $V$ in the aforementioned results, is the validity of the local Harnack inequality for positive solutions of (\ref{intro:Schrodinger}) and (\ref{intro:pLaplace}). For instance, Agmon's assumption on $V$ is optimal in the Lebesgue class of potentials for the Harnack inequality to be true. We stress that when the Harnack inequality fails, then the AP theorem might not be valid. Indeed, denote $p':=p/(p-1)$ the conjugate index of $p$, and suppose that $A$ is the identity matrix. Let $V\in\dD_{\rm loc}(\Om)$, where $\dD(\Om)$ is the dual of $\D(\Om)$ which is in turn defined as the closure of $\test(\Om)$ under the semi-norm $\|\nabla u\|_{L^p(\Om;\Rn)}$. If in addition to (\ref{intro:natural form bound}), one has that
 \be\nonumber
  \int_\Om\Big(|\nabla u|^p - k V|u|^p\Big)\dx\geq0
   \qquad \mbox{for all }u\in\test(\Om),
 \ee
 for some positive constant $k$, then the equation
 \be\label{intro:pLaplace:mv}
  -\Div\big(|\nabla v|^{p-2}\nabla v\big)+\alpha V|v|^{p-2}v=0
   \qquad \mbox{in }\Om,
 \ee
admits a positive solution (in a certain weak sense) for any $\alpha\in(0,p^\sharp)$, where $p^\sharp<1$ is given explicitly and depends only on $p$ (see \cite[Theorem 1.2 (i)]{JMV2}, or \cite[Theorem 1.1 (i)]{JMV1} for $p=2$). Moreover, this range for $\alpha$ is optimal as examples involving the Hardy potential reveals (see \cite[Remark 1.3]{JMV2}, or \cite[Example 7.3]{JMV1} for $p=2$). We note that under the above assumptions, the local Harnack inequality for positive solutions of \eqref{intro:pLaplace:mv} is in general not valid.

\medskip

The first aim of the present paper is to extend the AP theorem for the operator $Q'_{A,p,V}$ by relaxing significantly the condition $V\in L^\infty_{\rm loc}(\Om)$. In particular, under Agmon's (minimal) assumptions on the matrix $A$, we require $V$ to lie in a certain local Morrey space, the largest such that the Harnack inequality for positive solutions (and hence the local H\"{o}lder continuity of solutions) holds true. This means that we assume (see for instance \cite[\S5]{Tr}, \cite{RZ}, \cite{MZ} and also \cite{DF} for (\ref{intro:Schrodinger}))
 \be\label{intro:assumptions:V}
  \sup_{\substack{y\in\om\\0<r<\diam(\om)}}\varphi_q(r)\int_{\om\cap B_r(y)}|V|\dx<\infty
   \qquad \mbox{for all }\om\Subset\Om,
 \ee
where $\varphi_q(r)$ has the following behaviour near $0$
 \be\label{intro:def:varphi}
  \varphi_q(r)\;\underset{r\to 0}{\sim}\;\Bigg\{
  \begin{array}{llll}
    r^{-n(q-1)/q}\quad &\mbox{with }q>n/p & \mbox{ if } p<n,
     \\[2pt]
    \log^{q(n-1)/n}(1/r)\quad &\mbox{with }q>n & \mbox{ if } p=n,
     \\[2pt]
    1 & &\mbox{ if } p>n.
  \end{array}
 \ee
We prove in addition, that the assertions of the AP theorem are equivalent to the existence of a weak solution $T\in L^{p'}_{\rm loc}(\Om;\Rn)$ of the first order (nonlinear) divergence-type equation
 \be\nonumber
  -\Div(AT)+(p-1)|T|_A^{p'}=V.
 \ee
We refer to \cite[Theorem 1.3]{JMV1} for a related result with $A$ equals the identity matrix and $p=2$.

Recall that in general functions in Morrey spaces cannot be approximated by functions in $C^\infty(\Om)$, nor even by continuous functions (see \cite{Z}). Therefore, we cannot use an approximation argument to extend the AP theorem to our setting. Consequently, we need to start our study from the beginning of the topic and present in detail proofs involving new ideas.

\medskip

Another aim of the paper is to extend to the above class of operators several classical results and tools that hold true in {\em general} bounded domains (cf. \cite{AllH,GMdL,PR}, where stronger regularity assumptions on the coefficients and the boundary are assumed). In particular, we prove the existence of the principal eigenvalue, establish its main properties, and study the relationships between the positivity of principal eigenvalue, the weak and strong maximum principles, and the (unique) solvability of the Dirichlet problem.

\medskip

We then proceed to our main goal: establishing \textit{criticality theory} for (\ref{intro:pLaplace}) with $A$ and $V$ satisfying the above assumptions. To present the main results of the paper, let us recall that in case inequality (\ref{intro:natural form bound}) holds true but cannot be improved, in the sense that one cannot add on its right hand side a term of the form $\int_\Om W|u|^p\dx$ with a nonnegative function $W\not\equiv0$, then the nonnegative functional $Q_{A,p,V}$ is called \textit{critical} in $\Om$. Furthermore, a sequence $\{u_k\}_{k\in\N}\subset\wdtest(\Om)$ is called a {\em null sequence} with respect to the nonnegative functional $Q_{A,p,V}$ in $\Om$ if\vspace{.5em}

 a) $u_k\geq0$ for all $k\in\mathbb{N}$,
\vspace{.5em}

 b) there exists a fixed open set $K\Subset\Om$ such that $\|u_k\|_{L^p(K)}=1$ for all $k\in\mathbb{N}$.
\vspace{.5em}

 c) $\displaystyle{\lim_{k\to\infty}Q_{A,p,V}[u_k]=0}$,
\vspace{.5em}

\noindent A positive function  $\phi\in W_{\rm loc}^{1,p}(\Om)$ is called a {\em ground state} of $Q_{A,p,V}$ in $\Om$ if $\phi$ is an $L_{\rm loc}^p(\Om)$ limit of a null sequence. Finally, a positive solution $u$ of the equation $Q'_{A,p,V}[u]=0$ in $\Omega$ is a \emph{global minimal solution} if for any smooth compact subset $K$ of $\Om$, and any positive supersolution $v\in C(\Om\setminus{\rm int}K)$ of the equation $Q'_{A,p,V}[u]=0$ in $\Om\setminus K$, we have the implication
 \be\nonumber
  u\leq v\mbox{ on }\partial K\;\;\;\Rightarrow\;\;\;u\leq v\mbox{ in }\Om\setminus K.
 \ee
The central result of this paper is summarized in the following theorem.

\begin{theorem*}[Main Theorem]\label{main_thm} {\em Let $\Om$ be a domain in $\Rn$, where $n\geq2$, and suppose that the functional $Q_{A,p,V}$ is nonnegative on $\test(\Omega)$, where $A$ is a symmetric and locally uniformly positive definite matrix in $\Om$, and
\[\left\{
    \begin{array}{ll}
A\in L^\infty_{\rm loc}(\Om;\Rnn), \mbox { and } V \mbox{ satisfies (\ref{intro:assumptions:V}) with } \varphi_q \mbox{ as in } (\ref{intro:def:varphi})
 & \mbox { if } p\geq2,  \\[4mm]
A\in C^{0,\gamma}_{\rm loc}(\Om;\Rnn), \gamma\in(0,1),  \mbox { and } V \mbox{ satisfies  (\ref{intro:assumptions:V}) with } \varphi_q \;\underset{r\to 0}{\sim}\; r^q, q>n
  & \mbox{ if } p<2.
    \end{array}
  \right.
\]

\noindent Then the following assertions are equivalent:
\begin{enumerate}
  \item $Q_{A,p,V}$ is critical in $\Om$.
  \item $Q_{A,p,V}$ admits a null sequence in $\Om$.
  \item There exists a ground state $\phi$ which is a positive weak solution of \eqref{intro:pLaplace}.
  \item There exists a unique (up to a multiplicative constant) positive supersolution $v$ of \eqref{intro:pLaplace} in $\Om$.
  \item There exists a global minimal solution $u$ of \eqref{intro:pLaplace} in $\Om$.
  \end{enumerate}
In particular, $\phi=c_1v=c_2u$ for some positive constants $c_1,c_2$.

\medskip

Moreover, if $1<p\leq n$, then the above assertions are equivalent to
\begin{enumerate}
  \item[6.] Equation \eqref{intro:pLaplace} does not admit a positive minimal Green function.
\end{enumerate}}
\end{theorem*}

\begin{remark}
The additional regularity assumptions on $A$ and $V$ for the case $1<p<2$ in the Main Theorem seems to be technical, and might be nonessential. However, these assumptions guarantee the Lipschitz continuity of solutions of (\ref{intro:pLaplace}) (in fact they guarantee that solutions are $C^{1,\alpha}$, see \cite[Theorem 5.3]{Lb}), a property which (as in \cite{PT1,PR}) is essential for the proof of the Main Theorem in this range of $p$. On the other hand, throughout the paper we do not use the boundary point lemma, which was an essential tool in \cite{GMdL,PT1,PR}.
\end{remark}

\medskip

The structure of the article is presented next. In \S\ref{ssubsec:Morrey} we define the local Morrey space of potentials $V$  we are going to work with, and also present an uncertainty-type inequality for such potentials due to C.~B.~Morrey for $p=2$, and  D.~R.~Adams (see \cite[\S1.3]{MZ}) for $1<p<\infty$, that holds true in this space. This is the key property that is used in \cite{MZ,Tr} in order to extend Serrin's elliptic regularity theory \cite{Sr} for such equations. In \S\ref{ssubsec:(p,A)-Laplacian} we recall several well-known local regularity and compactness properties of (sub/super)solutions of equation (\ref{intro:pLaplace}) found in \cite{MZ} and \cite{PS}.

In \S\ref{sec:eigenvalue} we deal with bounded domains. Firstly, in \S\ref{subsec:prelims:eigenvalue} we establish some helpful lemmas, including the estimate \eqref{Ln} that extends to our case, a well-known inequality of P.~Lindqvist \cite{Ln} proved for the $p$-Laplace equation and concerns the positivity of the corresponding $I$ functional of Anane \cite{An} (see also Diaz and Saa \cite{DS}). We note that \eqref{Ln} replaces throughout our paper Picone's identity of Allegretto and Huang \cite{AllH}; a key tool in \cite{PT1,PR}. In addition, we prove in \S\ref{subsec:prelims:eigenvalue} the weak lower semicontinuity and the coercivity for two functionals related to the solvability of the Dirichlet problem in bounded domains. In \S\ref{subsec:eigenvalue} we use the results from \S\ref{subsec:prelims:eigenvalue} to prove the existence, simplicity and isolation of the principal eigenvalue $\lambda_1$ in a {\em general} bounded domain. Then we extend the main result in \cite{GMdL} concerning the equivalence of $\lambda_1$ being positive, the validity of the weak/strong maximum principle, and the existence of a unique positive solution for the Dirichlet problem
 \be\nonumber
  Q'_{A,p,V}[v]=g \;\;\mbox{ in } \om,\quad v\in \W(\om), \qquad  \mbox{where }g\in L^{p'}(p;\om)\mbox{ is nonnegative.}
 \ee

In passing from local to global, the results in bounded domains of \S\ref{sec:eigenvalue} are exploited in the last two sections. More precisely, in \S\ref{subsec:aap} we establish the AP theorem while in \S\ref{subsec:criticality} we prove among other results the equivalence of the first four statements of the Main Theorem. In addition,  we prove a Poincar\'{e}-type inequality for critical operators, and a Liouville comparison principle, generalizing results in \cite{PT1} and  \cite{P,PTT}, respectively (see also \cite{PR}).

The last two statements of the Main Theorem are treated in \S\ref{ssubsec:minimal} after establishing a suitable weak comparison principle (WCP) in \S\ref{ssubsec:prelims:minimal}, and the behaviour of positive solutions near an isolated singularity in \S\ref{ssubsec:isolated}.

We emphasize here, that generally speaking, we omit straightforward proofs that follow exactly the same steps as in the aforementioned papers, provided the needed tools have been obtained.
\section{Preliminaries}\label{sec:preliminaries}
In this section we fix our setting and notation, introduce some definitions, and review basic local regularity results of solutions of the equation (\ref{intro:pLaplace}).

\medskip

Throughout the paper we assume that
\begin{itemize}
  \item $1<p<\infty$.
  \item $\Om$ is a domain (an open and connected set) in $\Rn$, where $n\geq2$.
  \item $A=(a_{ij})\in L_{\rm loc}^{\infty}(\Om;\Rnn)$ is a \emph{symmetric} and \emph{locally uniformly positive definite} matrix.
\end{itemize}
The assumptions on $A$ imply in particular that
\begin{align*}\label{assumption:A}
\tag{S} & a_{ij}(x)=a_{ji}(x)\qquad \mbox{for a.e. } x\in\Om, \mbox{ and } i,j=1,...,n,  \\[0mm]
\tag{E} &  \forall\om\Subset\Om,\quad \exists\theta_\om>0~\mbox{ such that }~
 \theta_\om|\xi|\leq |\xi|_A
  \leq
   \theta_\om^{-1}|\xi|\qquad \mbox{for a.e. }x\in\om \mbox{ and } \forall\xi\in\Rn,
\end{align*}
where we have set
 \be\nonumber
  |\xi|_{A}
   :=
    \sqrt{A(x)\xi\cdot\xi}=\sqrt{\sum_{i,j=1}^na_{ij}(x)\xi_i\xi_j}
     \qquad \mbox{for a.e. }x\in\Om\mbox{ and }\xi=(\xi_1,...,\xi_n)\in\Rn.
 \ee
\noindent Moreover, we adopt the following notation:

\medskip

$q'$ is the conjugate index of $q\in(1,\infty)$, i.e. $q'=q/(q-1)$.

$\om\Subset\Om$ means $\om$ is a subdomain of $\Om$ with compact closure in $\Om$.

$B_r(y):=\{x\in\Rn:|x-y|<r\}$, where  $r>0,~y\in\Rn$.

$\Ln(E)$ is the Lebesgue measure of a measurable set $E\subset\Rn$.

$\langle f\rangle_\om$ is the mean value of a function $f$ in $\om$.

$\sprt\{f\}$ is the support of $f$.

$f^{+}:=\max\{f,0\}$, $f^-:=-\min\{f,0\}$ are the positive and negative parts of $f$, respectively.

$\gamma$ and $\gamma'$ will always stand for numbers in $(0,1)$.

$I_{n}$ is the identity matrix of size $n\times n$. 

$C(a,b,...)$ is a positive constant depending only on $a,b...~$, and may be different from line to line.
\subsection{Local Morrey spaces}\label{ssubsec:Morrey}
In the present subsection we introduce a certain class of Morrey spaces that depend on the index $p$, where $1<p<\infty$. It is the class of spaces where the potential $V$ of the operator $Q'_{A,p,V}$ belongs to.

\begin{definition}
Let $q\in[1,\infty]$ and $\om\Subset\Rn$. For a measurable, real valued function $f$ defined in $\om$, we set
 \be\nonumber
  \|f\|_{M^q(\om)}
   :=
    \sup_{\substack{y\in\om\\r<\diam(\om)}}\frac{1}{r^{n/q'}}\int_{\om\cap B_r(y)}|f|\dx.
 \ee
We write then $f\in M_{\rm loc}^q(\Om)$ if for any $\om\Subset\Om$ we have $\|f\|_{M^q(\om)}<\infty$.
\end{definition}

\begin{remark}\label{remark:Morrey} Note that $M_{\rm loc}^1(\Om)\equiv  L_{\rm loc}^1(\Om)$ and $M_{\rm loc}^\infty(\Om)\equiv L_{\rm loc}^\infty(\Om)$, but $L_{\rm loc}^q(\Om)\subsetneq M_{\rm loc}^q(\Om)\subsetneq L^1_{\rm loc}(\Om)$ for any $q\in(1,\infty)$.
\end{remark}

\noindent For the regularity theory of equations with coefficients in Morrey spaces we refer to the monographs \cite{MZ,M}, and also to the papers \cite{R} and \cite{BP} for further regularity issues. For generalizations of the Morrey spaces and other applications to analysis and systems of equations we refer to \cite{Ptr}, \cite{AdX1} and \cite{AdX2}.

\medskip

Next we define a special local Morrey space $M^q_{\rm loc}(p;\Om)$ which depends on the values of the exponent $p$.

\begin{definition}\label{def:q} For $p\neq n$, we define
 \be\nonumber
  M^q_{\rm loc}(p;\Om)
   :=
    \Bigg\{
     \begin{array}{ll}
      M^q_{\rm loc}(\Om)\mbox{ with }q>n/p & \mbox{if }p<n,
       \\[8pt]
      L^1_{\rm loc}(\Om) & \mbox{if }p>n,
     \end{array}
 \ee
while for $p=n$, $f\in M^q_{\rm loc}(n;\Om)$ means that for some $q>n$ and any $\om\Subset\Om$ we have
 \be\nonumber
  \|f\|_{M^q(n;\om)}
   :=
    \sup_{\substack{y\in\om\\0<r<\diam(\om)}}\varphi_q(r)\int_{\om\cap B_r(y)}|f|\dx
     <
      \infty,
 \ee
where $\varphi_q(r):=\log(\diam(\om)/r)^{q/n'}$ and $0<r<\diam(\om)$.
\end{definition}
In what follows we will frequently use the following key fact (sometimes called an uncertainty-type inequality) originally due to Morrey and further generalized by Adams (see \cite[Lemmas 5.2.1 \& 5.4.2]{M} for $p=2$, \cite[Lemma 5.1]{Tr} for $1<p<n,$ and \cite{RZ}, \cite[Corollary 1.95]{MZ}).

\begin{theorem}[Morrey-Adams theorem]\label{theorem:uncertainty}
Let $\om\Subset\Rn$, and suppose that $V\in M^q(p;\om)$.

$(i)$ There exists a constant $C(n,p,q)>0$ such that for any $\delta>0$ and all $u\in\W(\om)$
 \be\label{ineq:uncertainty}
  \int_\om|V||u|^p\dx
   \leq
    \delta\|\nabla u\|^p_{L^p(\om;\Rn)}
     +
      \frac{C(n,p,q)}{\delta^{n/(pq-n)}}\|V\|_{M^q(p;\om)}^{pq/(pq-n)}\|u\|^p_{L^p(\om)}.
 \ee

$(ii)$ For any $\om'\Subset\om$ with Lipschitz boundary there exist positive constant $C(n,p,q,\om',\om)$ and $\delta_0$ such that for any $0<\delta\leq \delta_0$ and all $u\in W^{1,p}(\om')$
 \be\nonumber
  \int_{\om'}|V||u|^p\dx
   \leq
    \delta\|\nabla u\|^p_{L^p(\om';\Rn)}
     +
       C(n,p,q,\delta,\|V\|_{M^q(p;\om)})\|u\|^p_{L^p(\om')}.
 \ee
\end{theorem}

\begin{proof} $(i)$ The case where $p\leq n$ is contained in \cite{MZ}. In particular, for $p<n$ this follows from \cite[Corollary 1.95]{MZ} (see also inequality (3.11) therein), while for $p=n$ one repeats that proof using \cite[Theorem 1.94]{MZ} instead of \cite[Theorem 1.93]{MZ}. Thus, we only need to argue for $p>n$. In this case our assumption reads $V\in L^1(\om)$. Recall also that by the Sobolev embedding theorem we have $\W(\om)\subset C(\om)$. It follows that
 \bea\nonumber & &
  \int_\om|V||u|^p\dx
   \leq
    \|V\|_{L^1(\om)}\|u\|^p_{L^\infty(\om)}
     \\ \nonumber & & \hspace{5.8em} \leq
      C(n,p)\|V\|_{L^1(\om)}\|\nabla u\|^{n}_{L^p(\om;\Rn)}\|u\|^{p-n}_{L^p(\om)},
 \eea
where we have used the Gagliardo-Nirenberg inequality (see for example \cite[Theorem 1.1 in \S IX]{DB}). The result follows by applying Young's inequality:
 \be\nonumber
  a b
   \leq
    \delta a^{p/n}+\frac{p-n}{p}\Big(\frac{n}{p\delta}\Big)^{n/(p-n)}b^{p/(p-n)},
 \ee
with $a=\|\nabla u\|^{n}_{L^p(\om)}$, $b=C(n,p)\|V\|_{L^1(\om)}\|u\|^{p-n}_{L^p(\om)}$.

$(ii)$ Let $\om'\Subset\om$ with $\partial\om'$ being Lipschitz. We may then consider the extension operator (see for example \cite[\S4.4]{EvG})
 \be\nonumber
  E:W^{1,p}(\om')\rightarrow\W(\om)
 \ee
such that for any $u\in W^{1,p}(\om')$ to have
 \be\label{exte}\left\{
  \begin{array}{lll}
   Eu=u\qquad \mbox{in }\om',
  \\[2mm]
   \|Eu\|_{L^p(\om)}\leq C(n,p,\om',\om)\|u\|_{L^p(\om')},
  \\[2mm]
   \|\nabla(Eu)\|_{L^p(\om;\Rn)}\leq C(n,p,\om',\om)\|u\|_{W^{1,p}(\om';\Rn)}.
  \end{array}\right.
 \ee
Thus, if $\delta>0$ and $u\in W^{1,p}(\om')$, it follows from \eqref{ineq:uncertainty} that
  \be\nonumber
  \int_{\om}|V||Eu|^p\dx
   \leq
    \delta\|\nabla(Eu)\|^p_{L^p(\om;\Rn)}
     +
      \frac{C(n,p,q)}{\delta^{n/(pq-n)}}\|V\|_{M^q(p;\om)}^{pq/(pq-n)}\|Eu\|^p_{L^p(\om)}.
 \ee
Applying (\ref{exte}) to the latter inequality yields $(ii)$. \end{proof}
\subsection{Regularity assumptions on $A$ and $V$}\label{ssubsec:regularity}
 We are now ready to introduce our regularity hypotheses on the coefficients of the operator $Q'_{A,p,V}$. Throughout the paper we assume that
\[\tag{H0}
 \mbox{the matrix }A\mbox{ satisfies }({\rm S}),~({\rm E}),\mbox{ and the potential }V\in M^q_{\rm loc}(p;\Om).
\]
In the sequel, in the case $1<p<2$, we sometimes make the following stronger hypothesis:

\[\tag{H1}
 A\in C^{0,\gamma}_{\rm loc}(\Om;\Rnn)\mbox{ satisfies }({\rm S}),~({\rm E}),\mbox{ and }V\in M^q_{\rm loc}(\Om),\mbox{ where }q>n.
\]
\subsection{The $(p,A)$-Laplacian with a potential term in $M^q_{\rm loc}(p;\Om)$}\label{ssubsec:(p,A)-Laplacian}
For a vector field $T\in L^1_{\rm loc}(\Om;\Rn)$ we define
 \be\nonumber
  \Div_AT
   :=
    \Div(AT),
 \ee
where $\Div(AT)$ is meant in the distributional sense.

In this paper we are interested in the $(p,A)$-Laplacian equation plus a potential term, that is
 \be\label{Q'=0}
  Q'_{A,p,V}[v] := -\Div_A(|\nabla v|_A^{p-2}\nabla v) + V|v|^{p-2}v  = 0\qquad \mbox{in }\Om.
 \ee
This is the Euler-Lagrange equation associated with the functional
 \be\label{Q}
  Q_{A,p,V}[u]
   :=
    \int_\Om\Big(|\nabla u|_A^p+V|u|^p\Big)\dx\qquad u\in\test(\Om).
 \ee

\begin{definition}\label{def.weak:sol:Q'=0}
Assume that $A$ and $V$ satisfy (H0). A function $v\in\Wloc(\Om)$ is a {\em solution} of (\ref{Q'=0}) in $\Om$ if
 \be\label{weak:sol:Q}
  \int_\Om|\nabla v|_A^{p-2}A\nabla v\cdot\nabla u\dx
   +\int_\Om V|v|^{p-2}vu\dx
    =0
     \qquad\mbox{for all }u\in\test(\Om),
 \ee
and a {\em (sub)supersolution} of (\ref{Q'=0}) in $\Om$ if
 \be\label{weak:ssol:Q}
  \int_\Om|\nabla v|_A^{p-2}A\nabla v\cdot\nabla u\dx
   +\int_\Om V|v|^{p-2}vu\dx
    \;(\leq)\geq0
     \qquad\mbox{for all nonnegative }u\in\test(\Om).
 \ee
A {\it strict supersolution} of (\ref{Q'=0}) in $\Om$ is a supersolution which is not a solution.
\end{definition}

\begin{remark}
The above definition makes sense because of condition $\rm{(E)}$, the Morrey-Adams theorem (Theorem~\ref{theorem:uncertainty}), and H\"{o}lder's inequality. In light of our assumptions on $A$ and $V$, and by a density argument, one can replace $\test(\Om)$ in Definition~\ref{def.weak:sol:Q'=0} by $\wtest(\Om)$, the space of all $L^p(\Om)$ functions having compact support in $\Om$ and first-order weak partial  derivatives in $L^p(\Om)$.
\end{remark}

\noindent The following theorem follows from \cite[Theorem 3.14]{MZ} for the case $p\leq n$, and from \cite[Theorem 7.4.1]{PS} for the case $p>n$.

\begin{theorem}[Harnack inequality]\label{theorem:harnack}\label{thrm:regularity}
Under hypothesis {\rm(H0)}, any nonnegative solution $v$ of $(\ref{Q'=0})$ in $\Om$ satisfies the local Harnack inequality. Namely, for any $\om'\Subset\om\Subset\Om$ there holds
 \be\label{Harnack}
  \sup_{\om'}v\leq C\inf_{\om'}v,
 \ee
where $C$ is a positive constant depending only on $n,p,q,\dist(\om',\om), \theta_\om$, and $\|V\|_{M^q(\om)}$ (and not on $v$).
\end{theorem}

\begin{remark}[Local H\"{o}lder continuity]\label{remark:regularity} A standard consequence of Theorem~\ref{theorem:harnack} is the following regularity assertion found in \cite[Theorem 4.11]{MZ} for $p\leq n$, and in \cite[Theorem 7.4.1 ]{PS} for $p>n$:

\medskip

{\it Under hypothesis {\rm(H0)}, any solution $v$ of $(\ref{Q'=0})$ in $\Om$ is locally H\"{o}lder continuous of order $\gamma$ (depending on $n,p,q$, and $\theta_\om$), and for any
$\om'\Subset\om\Subset\Om$, we have
\be\label{eq_local_holder}
  [v\,]_{\gamma,\om'} \leq C \sup_{\om}|v|,
 \ee
where $C$ is a positive constant depending only on $n,p,q,\dist(\om',\om)$,  $\theta_\om$,  and $\|V\|_{M^q(\om)}$. Here $[v\,]_{\gamma,\om'}$ is the H\"older seminorm of $v$ in $\om'$.
}
\end{remark}

\begin{remark}[Local Lipschitz continuity]\label{remark:local grad bound}
Later on, when proving Lemma~\ref{crit} for $p<2$, we will need conditions under which the local Lipschitz continuity of solutions is guaranteed. In other words, in the case $p<2$ we will need conditions that ensure the local boundedness of the modulus of the gradient of a solution of (\ref{Q'=0}). This and more are provided by \cite[Theorem 5.3]{Lb}:

\medskip

{\it Under hypothesis {\rm(H1)}, any solution $v$ of $(\ref{Q'=0})$ in $\Om$ is of class $C_{\rm loc}^{1,\gamma'}(\Om)$ for some $\gamma'\in(0,1)$ depending only on $n,p,\gamma$, $q$ and $\theta_\om$.}

\medskip

\noindent In particular, we will use the fact that whenever $\om'\Subset\om\Subset\Om$, then
\be\nonumber
  \sup_{\om'}|\nabla v|\leq C  \sup_{\om}|v|,
 \ee
for some positive constant $C$, depending only on $n,p,\gamma,q,\dist(\om',\om)$,  $\theta_\om$, $\|A\|_{C^{0,\gamma}(\om)}$,  and $\|V\|_{M^q(\om)}$.
\end{remark}

\begin{remark}[Weak Harnack inequality]\label{remark:harnack_supersol} For $p>n$, Theorem~\ref{theorem:harnack} holds true verbatim if $v$ is merely a nonnegative supersolution of (\ref{Q'=0}) in $\Om$ (see \cite[Theorem 7.4.1]{PS}). For $p\leq n$ we only have \cite[Theorem 3.13]{MZ}:

\medskip

{\it Let $p\leq n$ and set $s=n(p-1)/(n-p)$. Under hypothesis {\rm(H0)}, any nonnegative supersolution $v$ of $(\ref{Q'=0})$ in $\Om$ satisfies the weak Harnack inequality, namely, for any $\om'\Subset\om\Subset\Om$ and $0<t<s$ there holds
 \be\label{weakHarnack}
  \|v\|_{L^{t}(\om')}\leq C\inf_{\om'}v,
 \ee
where $C$ is a positive constant depending only on $n,p,t,\dist(\om',\om),\Ln(\om')$ and $\|V\|_{M^q(\om)}$.}
\end{remark}

We conclude the section with the following important result that will be used several times throughout the paper.
\begin{proposition}{\bf[Harnack convergence principle]}\label{Harnack:cp}
Consider a matrix $A\in L^\infty(\Om;\Rnn)$ which satisfies conditions ${\rm(A)}$ and ${\rm(E)}$. Let $\{\om_i\}_{i\in\N}$ be a sequence of Lipschitz domains such that $\om_i\Subset\Om$, $\om_i\Subset\om_{i+1}$ for $i\in\N$, and $\cup_{i\in\N}\om_i=\Om$, and fix a reference point $x_0\in \om_1$. Assume also that $\{\V_i\}_{i\in\N}\subset M^q(p;\om_i)$ converges in $M^q_{\rm loc}(p;\Om)$ to $\V\in M^q_{\rm loc}(p;\Om)$. For each $i\in\N$, let $v_i$ be a positive solution of the equation $Q'_{A_i,p,\V_i}[v]=0$ in $\om_i$  such that $v_i(x_0)=1$.

Then there exists then $0<\beta<1$, so that up to a subsequence, $\{v_i\}$ converges in $C^{0,\beta}_{\rm loc}(\Om)$ to a positive solution $v$ of the equation $Q'_{A,p,\V}[v]=0$ in $\Om$.
\end{proposition}

\begin{proof}The convergence in $C^{0,\beta}_{\rm loc}(\Om)$ follows by Arzel\`{a}-Ascoli theorem from the local Harnack inequality \eqref{Harnack}, and the local H\"{o}lder estimate \eqref{eq_local_holder}.

Now pick an arbitrary $\om\Subset\Om$. We will show that a subsequence of $\{v_i\}_{i\in\N}$ converges weakly in $W^{1,p}(\om)$ to a positive solution of  $Q'_{A,p,\V}[u]=0$ in $\Om$. Recall first that the definition of $v_i$ being a positive weak solutions to $Q'_{A,p,\V_i}[v]=0$ in $\om_i$ reads as
 \be\label{Q'=0_i}
  \int_{\om_i}|\nabla v_i|^{p-2}_AA\nabla v_i\cdot\nabla u\dx+\int_{\om_i}\V_iv_i^{p-1}u\dx=0
   \qquad \forall u\in\W(\om_i).
 \ee
By Remark \ref{remark:regularity}, $v_i$ are also continuous for all $i\in\N.$ Fix $k\in \mathbb{N}$. For $u\in\test(\om_k)$ we may thus pick $v_i|u|^p\in W_c^{1,p}(\om_k);~i\geq k$ as a test function in (\ref{Q'=0_i}) to get
 \be\nonumber
  \||\nabla v_i|_Au\|^p_{L^p(\om_k)}
   \leq
    p\int_{\om_k}|\nabla v_i|^{p-1}_A|u|^{p-1}v_i|\nabla u|_A\dx
     +
      \int_{\om_k}|\V_i|v_i^p|u|^p\dx.
 \ee
On the first term of the right hand side we apply Young's inequality: $pab\leq\e a^{p'}+[(p-1)/\e]^{p-1}b^p;$ $\e\in(0,1)$, with $a=|\nabla v_i|^{p-1}_A|u|^{p-1}$ and $b=v_i|\nabla u|_A$. On the second term we apply the Morrey-Adams theorem (Theorem~\ref{theorem:uncertainty}). We arrive at
 \bea\nonumber & &
  (1-\e)\||\nabla v_i|_Au\|^p_{L^p(\om_k)}
   \leq
    \big((p-1)/\e\big)^{p-1}\|v_i|\nabla u|_A\|^p_{L^p(\om_k)}
     \\ \nonumber &  & \hspace{12em} +
      \delta\|\nabla(v_iu)\|^p_{L^p(\om_k;\Rn)}
       +
         C(n,p,q,\delta,\|V\|_{M^q(p;\om_{k+1})})\|v_iu\|^p_{L^p(\om_k)}.
 \eea
By $({\rm E})$ and the simple fact that
 \be\nonumber
  \|\nabla(v_iu)\|^p_{L^p(\om_k;\Rn)}
   \leq
    2^{p-1}\big(\|v_i\nabla u\|^p_{L^p(\om_k;\Rn)}
     +
      \|u\nabla v_i\|^p_{L^p(\om_k;\Rn)}\big),
 \ee
we end up with the following Caccioppoli estimate valid for all $i\geq k$ and any $u\in\test(\om_k)$
 \begin{align}\nonumber
  &\Big((1-\e)\theta^p_{\om_k}-2^{p-1}\delta\theta^{-p}_{\om_k}\Big)\||\nabla v_i|u\|^p_{L^p(\om_k)}\\[4mm] \label{Cacc}
   & \leq
    \Big(\big(p\!-\!1)/\e\big)^{p-1}\theta^{-p}_{\om_k}\!+\!2^{p-1}\delta\Big)\|v_i|\nabla u|\|^p_{L^p(\om_k)}
             \!+\!
       C(n,p,q,\delta,\|V\|_{M^q(p;\om_{k+1})})\|v_iu\|^p_{L^p(\om_k)}.
 \end{align}
Without loss of generality we assume that $\om$ contains $x_0$. Picking $\om'\Subset\Om$ such that $\om\subset\om'$, we find $k\geq 1$ such that $\om'\subset\om_k$. Next we chose $\delta<(1-\e)2^{1-p}\theta^{2p}_{\om_k}$ and specialize $u\in\test(\om_k)$ such that
\be\label{test:function}
  \sprt\{u\}\subset\om',\quad0\leq u\leq1\mbox{ in }\om',\quad u=1\mbox{ in }\om\quad\mbox{ and }\quad|\nabla u|\leq1/\dist(\om',\om)\mbox{ in }\om.
\ee
Applying this to the Caccioppoli inequality \eqref{Cacc}, and using the fact that $\{v_i\}_{i\in\N}$ is bounded in the $L^\infty(\om)$-norm uniformly in $i$ (due to the local Harnack's inequality \eqref{Harnack}), we conclude
 \be\nonumber
  \|\nabla v_i\|^p_{L^p(\om;\Rn)}
   +
    \|v_i\|^p_{L^p(\om)}
     \leq
      C\big(n,p,q,\e,\delta,\dist(\om',\om),\theta_{\om_k},\|\V\|_{M^q(p;\om_{k+1})}\big)
       \qquad \mbox{for all }i\geq k.
 \ee
So $\{v_i\}_{i\in\N}$ is bounded in the $W^{1,p}(\om)$. By weak compactness of $W^{1,p}(\om)$, there exists a subsequence, still denoted by $\{v_i\}_{i\in\N}$, that converges weakly in $W^{1,p}(\om)$ to a nonnegative function $v$ with $v(x_0)=1$.

Next we show that $v$ is a solution of $Q'_{A,p,\V}[u]=0$ in $\tilde{\om}\Subset\om$ such that $x_0\in\tilde{\om}$. First note that for a subsequence (that once more we do not rename)  we have $v_i\rightarrow v$ a.e. in $\om$ and in $L^p(\om)$. For the potential term of the equation we note first that (up to a subsequence) $\V_i\rightarrow\V$ a.e. in $\om$. Thus, $\V_iv_i^{p-1}\rightarrow\V v^{p-1}$ a.e. in $\om$, while $|\V_iv_i^{p-1}|\leq c|\V|$ a.e. in $\om,$ where $c$ is independent of $i$. Since $|\V|\in M^q_{\rm loc}(p;\Om)\subset L^1_{\rm loc}(\Om)$ we may apply the dominated convergence theorem to get
\be\label{weak:conv:perturb}
  \int_\om\V_iv_i^{p-1}u\dx
   \rightarrow
    \int_\om\V v^{p-1}u\dx\qquad \mbox{for all }u\in\test(\om).
 \ee
It remains to prove that
\be\label{weak:conv:principal}
 \xi_i:=|\nabla v_i|^{p-2}_AA\nabla v_i\underset{i \to \infty}{\rightharpoonup} |\nabla v|^{p-2}_AA\nabla v=:\xi\qquad\mbox{in }L^{p'}(\tilde{\om};\Rn).
\ee
To this end, letting $u$ be as in (\ref{test:function}) but with $\om,\om'$ replaced by $\tilde{\om},\om$ respectively, we take $u(v_i-v)$ as a test function in (\ref{Q'=0_i}), to obtain
 \be\label{in9}
  \int_{\om}u\xi_i\cdot\nabla(v_i-v)\dx
   =
    -\int_{\om}(v_i-v)\xi_i\nabla u\dx
     -\int_{\om}\V_iv_i^{p-1}u(v_i-v)\dx.
 \ee
We claim that
\begin{equation}\label{eq1}
  \int_{\om}u\xi_i\cdot\nabla(v_i-v)\dx\underset{i \to \infty}{\longrightarrow}0.
\end{equation}

Indeed, by an argument similar to the one leading to (\ref{weak:conv:perturb}), the second integral on the right of \eqref{in9} converges to $0$ as $i\to\infty.$ For the first one, apply Holder's inequality to get
\bea\nonumber & &
 \Big|-\int_{\om}(v_i-v)\xi_i\nabla u\dx\Big|
  \leq
   \theta_\om^{p/p'}\|(v_i-v)\nabla u\|_{L^p(\om;\Rn)}\|\nabla v_i\|^{p/p'}_{L^p(\om;\Rn)}
    \\ \nonumber & & \hspace{10.3em} \leq
     C\big(p,\theta_\om,\dist(\tilde{\om},\om)\big)\|v_i-v\|_{L^p(\om)}\|\nabla v_i\|^{p/p'}_{L^p(\om;\Rn)},
\eea
which also converges to $0$ as $i\to\infty$ since $\|\nabla v_i\|_{L^p(\om;\Rn)}$ are uniformly bounded and $v_i\rightarrow v$ in $L^p(\om).$

Notice that as in the case where $A=I_n$, we have for any $X,~Y\in\Rn;~n\geq1$,
 \bea\nonumber & &
  \big(|X|^{p-2}_AAX-|Y|^{p-2}_AAY\big)\cdot(X-Y)
   =
    |X|^p_A-|X|^{p-2}_AAX\cdot Y+|Y|^p_A-|Y|^{p-2}_AAY\cdot X
     \\ \nonumber & & \hspace{15.8em} \geq
      |X|^p_A-|X|^{p-1}_A|Y|_A+|Y|^p_A-|Y|^{p-1}_A|X|_A
       \\ \nonumber & & \hspace{15.8em} =
        \big(|X|^{p-1}_A-|Y|^{p-1}_A\big)(|X|_A-|Y|_A)
         \\ \label{cauchy-schwartz} & & \hspace{15.8em} \geq
          0
 \eea
The above considerations imply that
\be\nonumber
 0\leq
  \mathcal{I}_i
   :=
    \int_{\tilde{\om}}(\xi_i-\xi)\cdot\nabla(v_i-v)\dx      \leq
      \int_{\om}u(\xi_i-\xi)\cdot\nabla(v_i-v)\dx \underset{i \to \infty}{\longrightarrow}
        0,
\ee
where we have used \eqref{eq1} and the weak convergence in $L^{p'}(\om;\Rn)$ of $\nabla v_i$ to $\nabla v$. Thus $\lim_{i\to\infty}\mathcal{I}_i=0$ and invoking a celebrated Lemma of Maz'ya \cite{Mz} (see also Lemma 3.73 of \cite{HKM}), \eqref{weak:conv:principal} follows.

\medskip

Hence, using Harnack's inequality, we have that $v$ is a positive weak solution of $Q'_{A,p,\V}[u]=0$ in $\tilde{\om}$ with $v(x_0)=1$. We now use a standard Harnack chain argument and a diagonalization procedure to obtain a new subsequence (once again not renamed) $\{v_i\}_{i\in\N}$, such that $v_i\rightharpoonup v$ in $W_{\rm loc}^{1,p}(\Om)$ (and locally uniformly in $\Om$), where $v$ is a positive weak solution of $Q'_{A,p,\V}[u]=0$ in $\Om$.
\end{proof}
\section{Principal eigenvalue and the maximum principle}\label{sec:eigenvalue}
Throughout the present section we fix a bounded domain $\om$ in $\Rn$, and suppose that $A$ is a uniformly elliptic, bounded matrix in $\om$, and $V\in M^q(p;\om)$. We consider in $\om$ the operator $Q'_{A,p,V}$ defined in (\ref{Q'=0}), and for $u\in\test(\om)$ we denote \[Q_{A,p,V}[u;\om]:= \int_\om\Big(|\nabla u|^{p}_A+V(x)|u|^p\Big)\dx.\]
 \begin{definition}
We say that $\lambda\in\R$ is an {\it eigenvalue with an eigenfunction} $v$ of the Dirichlet eigenvalue problem
 \be\label{eigenproblem}
  \left\{
    \begin{array}{ll}
      Q'_{A,p,V}[w]=\lambda|w|^{p-2}w & \mbox{in }\om, \\
       w=0 & \mbox{ on }\partial\om,
    \end{array}
  \right.
     \ee
if $v\in\W(\om)\setminus\{0\}$ satisfies
 \be\label{eigendef}
  \int_\om|\nabla v|_A^{p-2}A\nabla v\cdot\nabla u\dx
   +
    \int_\om V|v|^{p-2}vu\dx
     =
      \lambda\int_\om|v|^{p-2}vu\dx
       \qquad \mbox{for all }u\in\test(\om).
 \ee
\end{definition}
\begin{definition}
A {\it principal eigenvalue} is an eigenvalue of \eqref {eigenproblem} with a nonnegative eigenfunction.
\end{definition}

\noindent The existence of a principal eigenvalue for the problem (\ref{eigenproblem}), and its variational characterization by the Rayleigh-Ritz variational formula
 \be\label{first:eigenvalue}
  \lambda_1=\lambda_1(Q_{A,p,V};\om)
    :=\inf_{u\in \W(\om)\setminus\{0\}}\dfrac{Q_{A,p,V}[u;\om]}{\|u\|^p_{L^p(\om)}}\,,
 \ee
is established in Proposition~\ref{proposition:principal} below.

\medskip

Consider first the equation
 \be\label{Q'=f:om}
  Q'_{A,p,V}[v]=g \quad \mbox{in } \omega, \qquad \mbox{where }g\in M^q(p;\om)\mbox{ is nonnegative}.
 \ee
By a (sub, super)solution of (\ref{Q'=f:om}) we mean a function $v\in W^{1,p}_{\mathrm{loc}}(\om)$ such that
 \bea\nonumber
  \int_\om\!\!|\nabla v|_A^{p-2}A\nabla v\!\cdot\!\nabla u\!\dx
   +
    \int_\om \!\!V|v|^{p-2}vu\!\dx
     (\leq,~\geq)\!=\!
       \int_\om\!\! gu\!\dx
        \quad \mbox{for all (nonnegative) }u\!\in\!\test(\om).
 \eea
One of our targets in the following subsection is to characterize in terms of the strict positivity of the principal eigenvalue of problem (\ref{eigenproblem}), the following properties
\begin{itemize}
  \item[a)] the solvability in $\W(\om)$ of (\ref{Q'=f:om}),
  \item[b)] the (generalized) weak maximum principle for (\ref{Q'=f:om}),
  \item[c)] the strong maximum principle for (\ref{Q'=f:om}).
\end{itemize}
Recall at this point that the {\it (generalized) weak maximum principle} for the operator $ Q'_{A,p,V}$ asserts that a solution of the equation \eqref{Q'=f:om} which is nonnegative on $\partial\om$ is nonnegative in $\om$, while the {\it strong maximum principle} asserts that in addition to the weak maximum principle, a solution of (\ref{Q'=f:om}) which is nonnegative on $\partial\om$, is either identically zero or strictly positive in $\om$.
\subsection{Preparatory material}\label{subsec:prelims:eigenvalue}
We start with the following technical lemma that generalizes computations found in \cite{An,DS,Ln}, where the case $V_1=V_2\equiv0$ and $A=I_n$ is considered. This useful lemma replaces Picone's identity which is a key tool in \cite{PT1,PR}.  We note that in the present paper the lemma is used only for the case $V_1=V_2$, but this assumption does not affect at all the volume of computations of the general case.

\begin{lemma}\label{lemma:lind} Let $g_i,~V_i\in M^q(p;\om)$, where  $i=1,2$. There exists a positive constant $c_p$, depending only on $p$  such that the following assertions holds true:

(i)  Suppose that $w_1,~w_2\in\W(\om)\setminus\{0\}$ are nonnegative solutions of
 \be\label{solutions:equations}
  Q'_{A,p,V_1}[w;\om]=g_1,\qquad\mbox{ and }\qquad Q'_{A,p,V_2}[w;\om]=g_2,
 \ee
respectively, and let $w_{i,h}:=w_i+h$, where $h$ is a positive constant, and $i=1,2$. Then
\bea\nonumber & &
 \hspace{-2em}I_h:=\int_\om\Big(\frac{g_1-V_1w^{p-1}_1}{{w^{p-1}_{1,h}}}-\frac{g_2-V_2w_2^{p-1}}{{w^{p-1}_{2,h}}}\Big)(w^p_{1,h}-w^p_{2,h})\dx
  \\ \label{Ln} & &   \geq
   c_p\left\{
     \begin{array}{ll}
      \!\!\displaystyle{\int_\om(w^p_{1,h}+w^p_{2,h})\Big|\nabla\log\frac{w_{1,h}}{w_{2,h}}\Big|_A^p\dx} & \mbox{if }p\geq2,
       \\[6mm]
      \!\!\displaystyle{\int_\om \!\!(w^p_{1,h}\!+\!w^p_{2,h})\Big|\nabla\log\frac{w_{1,h}}{w_{2,h}}\Big|_A^2\big(|\nabla\log w_{1,h}|_A\!+\!|\nabla\log w_{2,h}|_A\big)^{p-2}\!\dx} & \mbox{if }p<2.
     \end{array}\right.
\eea

(ii) In the particular case of nonnegative eigenfunctions, i.e.,
$$w_1:=w_\lambda,\quad w_2:= w_\mu, \quad g_1:=\lambda|w_\lambda|^{p-2}w_\lambda, \quad g_2=\mu|w_\mu|^{p-2}w_\mu,$$
with  $\lambda,~\mu\in\R$, we have
\bea\nonumber & &
 \hspace{-2em}\int_\om\Big((\lambda-\mu)-(V_1-V_2)\Big)(w^p_\lambda-w^p_\mu)\dx
  \\ \nonumber & & \hspace{2em} \geq
   c_p\left\{
     \begin{array}{ll}
      \displaystyle{\int_\om(w^p_\lambda+w^p_\mu)\Big|\nabla\log\frac{w_\lambda}{w_\mu}\Big|_A^p\dx} & \mbox{if }p\geq2,
       \\[6mm]
      \displaystyle{\int_\om(w^p_\lambda+w^p_\mu)\Big|\nabla\log\frac{w_\lambda}{w_\mu}\Big|_A^2\big(|\nabla\log w_\lambda|_A+|\nabla\log w_\mu|_A\big)^{p-2}\dx} & \mbox{if }p<2.
     \end{array}\right.
\eea

(iii) Suppose further that $\om$ is Lipschitz, and let  $w_1,~w_2\in W^{1,p}(\om)$ be positive solutions of (\ref{solutions:equations}) respectively, such that $w_1=w_2>0$ on $\partial\om$, in the trace sense. Then
 \bea\nonumber & &
  \hspace{-2em}\int_\om\Big(\frac{g_1}{w^{p-1}_1}-\frac{g_2}{w^{p-1}_2}-(V_1-V_2)\Big)(w^p_1-w^p_2)\dx
   \\ \nonumber & & \geq
    c_p\left\{
     \begin{array}{ll}
      \displaystyle{\int_\om(w^p_1+w^p_2)\Big|\nabla\log\frac{w_1}{w_2}\Big|_A^p\dx} & \mbox{if }p\geq2,
       \\[6mm]
      \displaystyle{\int_\om(w^p_1+w^p_2)\Big|\nabla\log\frac{w_1}{w_2}\Big|_A^2\big(|\nabla\log w_1|_A+|\nabla\log w_2|_A\big)^{p-2}\dx} & \mbox{if }p<2.
     \end{array}\right.
 \eea
\end{lemma}

\begin{proof} Set $\psi_{1,h}:=(w^p_{1,h}-w^p_{2,h})w^{1-p}_{1,h}$. It is easily seen that $\psi_{1,h}\in\W(\om)$, and using it as a test function in the definition of $w_1$ being a solution of the first equation of (\ref{solutions:equations}), we get
 \bea\nonumber & & \hspace{-2em}
  \int_\om(w^p_{1,h}-w^p_{2,h})|\nabla(\log w_{1,h})|_A^p\dx
   -
    p\int_\om w^p_{2,h}|\nabla(\log w_{1,h})|_A^{p-2}A\nabla(\log w_{1,h})\cdot\nabla\Big(\log\frac{w_{2,h}}{w_{1,h}}\Big)\dx
     \\ \nonumber & & =
      \int_\om\frac{g_1-V_1w^{p-1}_1}{{w^{p-1}_{1,h}}}(w^p_{1,h}-w^p_{2,h})\dx.
 \eea
In the same fashion we set $\psi_{2,h}:=(w^p_{2,h}-w^p_{1,h})w^{1-p}_{2,h}$ and use it as a test function in the definition of $w_2$ being a solution of the second equation of (\ref{solutions:equations}), to obtain
 \bea\nonumber & & \hspace{-2em}
  \int_\om(w^p_{2,h}-w^p_{1,h})|\nabla(\log w_{2,h})|_A^p\dx
   -
    p\int_\om w^p_{1,h}|\nabla(\log w_{2,h})|_A^{p-2}A\nabla(\log w_{2,h})\cdot\nabla\Big(\log\frac{w_{1,h}}{w_{2,h}}\Big)\dx
     \\ \nonumber & & =
      \int_\om\frac{g_2-V_2w_2^{p-1}}{{w^{p-1}_{2,h}}}(w^p_{2,h}-w^p_{1,h})\dx.
 \eea
Adding these we arrive at
 \bea\nonumber & & \hspace{-2em}
  \int_\om w^p_{1,h}\Big[|\nabla(\log w_{1,h})|_A^p-|\nabla(\log w_{2,h})|_A^p-p|\nabla(\log w_{2,h})|_A^{p-2}A\nabla(\log w_{2,h})\cdot\nabla\Big(\log\frac{w_{1,h}}{w_{2,h}}\Big)\Big]\dx
   \\ \nonumber & & \hspace{-2em} +
    \int_\om w^p_{2,h}\Big[|\nabla(\log w_{2,h})|_A^p-|\nabla(\log w_{1,h})|_A^p-p|\nabla(\log w_{1,h})|_A^{p-2}A\nabla(\log w_{1,h})\cdot\nabla\Big(\log\frac{w_{2,h}}{w_{1,h}}\Big)\Big]\dx
     \\ \label{pre:vecineq} & & =
      I_h.
 \eea
Now we use the following inequality found in \cite[Lemma 4.2]{Ln} for $A$ being the identity matrix $I_n$, cf. \cite[(2.19)]{PTT} (the proof is essentially the same and we omit it): {\it for all vectors $\alpha,\beta\in\Rn$ and a.e. $x\in\om$, we have}
 \be\label{vecineqLindqvist}
  |\alpha|_A^p-|\beta|_A^p-p|\beta|_A^{p-2}A(x)\beta\cdot(\alpha-\beta)
  \geq
   C(p)\Bigg\{
     \begin{array}{ll}
      |\alpha-\beta|_A^p & \mbox{if }p\geq2,
       \\[8pt]
      |\alpha-\beta|_A^2(|\alpha|_A+|\beta|_A)^{p-2} & \mbox{if }p<2.
     \end{array}
 \ee
Applying this to both terms of the left hand side of (\ref{pre:vecineq}), we obtain the inequality of part ($i$).

To prove part ($ii$), take $g_1=\lambda|w_1|^{p-2}w_1,~g_2=\mu|w_2|^{p-2}w_2$ for some $\lambda,~\mu\in\R$, and rename $w_1,~w_2$ to $w_\lambda,~w_\mu$ respectively. The integrand of $I_h$ in this case satisfies for all $0<h<1$
 \bea\nonumber & & \hspace{-2em}
  \Big|\Big[(\lambda-V_1)\Big(\frac{w_\lambda}{w_{\lambda,h}}\Big)^{p-1}-(\mu-V_2)\Big(\frac{w_\mu}{w_{\mu,h}}\Big)^{p-1}\Big](w^p_{\lambda,h}-w^p_{\mu,h})\Big|
   \\ \nonumber & & \leq
    (|\lambda-V_1|+|\mu-V_2|)[(w_\lambda+1)^p+(w_\mu+1)^p]\in L^1(\om),
 \eea
by Theorem~\ref{theorem:uncertainty}-($i$). As $h\to 0$,  we have
$$\Big[(\lambda-V_1)\Big(\frac{w_\lambda}{w_{\lambda,h}}\Big)^{p-1}-(\mu-V_2)\Big(\frac{w_\mu}{w_{\mu,h}}\Big)^{p-1}\Big](w^p_{\lambda,h}-w^p_{\mu,h}) \to (\lambda-\mu-V_1+V_2)(w^p_\lambda-w^p_\mu)$$ a.e. in $\om$.  By applying the dominated convergence theorem and the Fatou lemma on the inequality of part ($i$), we get the desired estimate. Part ($iii$) follows from part ($i$) by setting $h=0$.
\end{proof}

We modify to our case a well known lemma on the negative part of a supersolution (see for example, \cite[Lemma 2.7]{Agm}, or \cite[Lemma 2.4]{PTT}).

\begin{lemma}\label{lemma:Kato}
Let $\V\in M_{\rm loc}^q(p;\Om).$ If $v\in\Wloc(\Om)$ is a supersolution of $Q'_{A,p,\V}[u]=0$ in $\Om,$ then $v^-$ is a $\Wloc(\Om)$ subsolution of the same equation.
\end{lemma}

\begin{proof}
Though this argument is quite standard, we add it for completeness, and since it requires the use of the Morrey-Adams theorem in the final limit argument. Following the steps of the proof in \cite{Agm}, we define
\[
\varphi_\e:=\frac{v_\e-v}{2v_\e}\varphi\qquad\mbox{and}\qquad v_\e:=(v^2+\e^2)^{1/2},
\]
with $\varphi$ being an arbitrary nonnegative function in $\test(\Om)$. It is straightforward to see that
\[
\nabla v_\e\cdot\nabla\varphi\leq\nabla v\cdot\nabla\Big(\frac{v}{v_\e}\varphi\Big)\qquad\mbox{a.e. in }\Om,
\]
and then
\be\label{varphi}
\frac{1}{2}\nabla(v_\e-v)\cdot\nabla\varphi\leq-\nabla v_\e\cdot\nabla\varphi_\e\qquad\mbox{a.e. in }\Om.
\ee
Thus, taking $\varphi_\e\in W_c^{1,p}(\Om)$ as a test function in the definition of $v\in\Wloc(\Om)$ being a supersolution of $Q'_{A,p,\V}[u]=0$ in $\Om,$ and then applying (\ref{varphi}), we conclude that we only need to show that we can take the limit $\e\to0,$ in the following expression
\be\label{kato:agmon}
\frac{1}{2}\int_\Om|\nabla v|_A^{p-2}A\nabla\big(v_\e-v\big)\cdot\nabla\varphi\dx
 -
  \int_\Om\V|v|^{p-2}v\varphi_\e\dx
   \leq0.
\ee
Note that since $\nabla\big(v_\e-v\big)/2\rightarrow \nabla v^-$, and $v\varphi_\e\rightarrow-v^-\varphi$ as $\e\to0$, this would readily give
\[
\int_\Om|\nabla v^-|_A^{p-2}A\nabla v^-\cdot\nabla\varphi\dx
 +
  \int_\Om\V|v^-|^{p-2}v^-\varphi\dx
   \leq0,\qquad \mbox{for all nonnegative }\varphi\in\test(\Om).
\]
However, the justification of taking the limit inside both integrals in (\ref{kato:agmon}) is verified by the dominated convergence theorem. For the first one we use H\"{o}lder's inequality, while for the second we apply first H\"{o}lder's inequality and then the Morrey-Adams theorem.
\end{proof}

\begin{definition}
Let $(X, \|\cdot\|_X)$ be a Banach space. A functional $J:X\rightarrow\R\cup\{\infty\}$ is said to be {\em coercive} if $J[u]\to \infty$ as $\|u\|_X\to\infty$.  The functional $J$ is said to be {\it(sequentially) weakly lower semicontinuous} if $J[u]\leq\displaystyle{\liminf_{k\to\infty}}J[u_k]$ whenever $u_k\rightharpoonup u$.
\end{definition}
We have
\begin{proposition}\label{proposition:wlsc}
(a) Let $\om\Subset\Rn$, $\V\in M^q(p;\om)$ and  $\G\in L^{p'}(\om)$. Define the functional $J:\W(\om)\to\R\cup\{\infty\}$ by
\be\label{funct:J}
 J[u]
  :=
   Q_{A,p,\V}[u;\om]-\int_\om\G u\dx.
\ee
Then $J$ is weakly lower semicontinuous in $\W(\om)$.

\medskip

(b) Let $\om\Subset \om'\Subset \R^n$ with $\om$ being Lipschitz, and let  $\G,~\V\in M^q(p;\om')$. Define the functional $\bar{J}:W^{1,p}(\om)\to\R\cup\{\infty\}$ by
\be\label{funct:barJ}
 \bar{J}[u]
  :=
   Q_{A,p,\V}[u;\om]-\int_\om\G|u|\dx.
 \ee
Then $\bar{J}$ is weakly lower semicontinuous in $W^{1,p}(\om)$.
\end{proposition}

\begin{proof} We first prove statement ($b$). Let $u,\{u_k\}_{k\in\N}\subset W^{1,p}(\om)$ be such that $u_k\rightharpoonup u$ in $W^{1,p}(\om)$. By the uniform boundedness principle, we have
\[
 K:=\sup_{k\in\N}\|u_k\|_{W^{1,p}(\om)}<\infty,
\]
and thus by the compact imbedding of $W^{1,p}(\om)$ in $L^p(\om)$, and by passing to a subsequence, we may assume that $u_k\rightarrow u$ in $L^p(\om)$ and a.e. in $\om$. 

Let $\delta>0$. By Minkowski's inequality and the Morrey-Adams theorem (Theorem~\ref{theorem:uncertainty}-($ii$)), we have
\bea\nonumber & &
  \Big(\int_\om \V^\pm|u_k|^p\dx\Big)^{1/p}-\Big(\int_\om \V^\pm|u|^p\dx\Big)^{1/p}
   \leq
    \Big(\int_\om \V^\pm|u_k-u|^p\dx\Big)^{1/p}
     \\ \label{replace:ineq:trace} & & \hspace{2em} \leq
      \Big(\delta\|\nabla(u_k-u)\|^p_{L^p(\om;\Rn)}+C(n,p,q,\delta,\|\V^\pm\|_{M^q(p;\om')})\|u_k-u\|^p_{L^p(\om)}\Big)^{1/p}
       \\ \nonumber & & \hspace{2em} \leq
        \delta^{1/p}(K+\|\nabla u\|_{L^p(\om;\Rn)})+C(n,p,q,\delta,\|\V^\pm\|_{M^q(p;\om')})\|u_k-u\|_{L^p(\om)}.
\eea
This shows that
 \be\nonumber
  \limsup_{k\to\infty}\int_\om\V^\pm|u_k|^p\dx
   \leq
    \int_\om\V^\pm|u|^p\dx.
 \ee
On the other hand, by Fatou's Lemma, we have
 \be\nonumber
  \int_\om\V^\pm|u|^p\dx
   \leq
    \liminf_{k\to\infty}\int_\om\V^\pm|u_k|^p\dx.
 \ee
The last two inequalities imply
 \be\nonumber
  \lim_{k\to\infty}\int_\om\V|u_k|^p\dx
   =
    \int_\om\V|u|^p\dx,
 \ee
The weak lower semicontinuity of the gradient term follows from the convexity of the Lagrangian $\zeta\mapsto |\zeta|_{A(x)}^p$. We deduce then
\be\label{wlsc:Q'}
 Q_{A,p,\V}[u]\leq\liminf_{k\to\infty}Q_{A,p,\V}[u_k].
\ee
For the last term of $J$, we work similarly
\bea\nonumber &  &
 \int_\om\G^\pm|u_k|\dx-\int_\om\G^\pm|u|\dx
  \leq
   \|\G^\pm\|^{1/p'}_{L^1(\om)}\Big(\int_\om\G^\pm|u_k-u|^p\dx\Big)^{1/p}
    \\ \nonumber & & \qquad  \leq
     \delta^{1/p}\|\G^\pm\|^{1/p'}_{L^1(\om)}\big(K+\|\nabla u\|_{L^p(\om;\Rn)}\big)+C(n,p,q,\delta,\|\G^\pm\|_{M^q(p;\om')})\|u_k-u\|_{L^p(\om)},
\eea
and thus
\be\nonumber
 \limsup_{k\to\infty}\int_\om\G^\pm|u_k|\dx
  \leq
   \int_\om\G^\pm|u|\dx.
\ee
On the other hand,
\be\nonumber
 \int_\om\G^\pm|u|\dx
  \leq
   \liminf_{k\to\infty}\int_\om\G^\pm|u_k|\dx.
\ee
The last two inequalities imply
\be\nonumber
 \lim_{k\to\infty}\int_\om\G|u_k|\dx
  =
   \int_\om\G|u|\dx,
\ee
and thus $\bar{J}$ is weakly lower semicontinuous in $W^{1,p}(\om)$.

For the proof of the weak lower semicontinuity of $J$ in $\W(\om)$, one follows the same steps, but uses Theorem~\ref{theorem:uncertainty}-($i$) in (\ref{replace:ineq:trace}), in order to obtain (\ref{wlsc:Q'}). Note that since we require in this case that $\G\in L^{p'}(\om)$, the functional  $I(u):=\int_\om\G u\dx$ is weakly continuous since it is a bounded linear functional.
\end{proof}

\begin{proposition}\label{proposition:coercivity}
(a)  Let $\om\Subset\om'\Subset  \R^n$, where $\om$ is Lipschitz,  and $\G,~\V\in M^q(p;\om')$.  If $\,\V$ is nonnegative, then for any $f\in W^{1,p}(\om)$ we have that $\bar{J}$ is coercive in \[\mathcal{A}:=\{u\in W^{1,p}(\om)\mbox{ s.t. }u=f\mbox{ on }\partial\om\}.\]
(b) Let $\om\Subset \R^n$, $\V\in M^q(p;\om)$ and $\G\in L^{p'}(\om)$. Assume that for some $\e>0$ we have
 \be\label{subcriticality}
  Q_{A,p,\V}[u;\om]
   \geq
    \e\|u\|_{L^p(\om)}^p
     \qquad \mbox{for all }u\in\W(\om).
 \ee
Then $J$ is coercive in $\W(\om)$.
\end{proposition}

\begin{proof} (a) Fix $t\in\R$, and suppose that $u\in\mathcal{A}$ is such that $\bar{J}[u]\leq t$. It is enough to prove that
 \be\label{estim:W:crcv}
  \|u\|_{W^{1,p}(\om)}:=\|u\|_{L^p(\om)}+\|\nabla u\|_{L^p(\om;\Rn)}
   \leq
    C,
 \ee
with $C$ independent of $u$. To this end, from $\bar{J}[u]\leq t$ and since $\V\geq0$ a.e. in $\om$, we readily deduce
\bea\nonumber & &
 \int_\om|\nabla u|_A^p\dx
  \leq
   t+\int_\om\G|u|\dx
    \\ \nonumber & & \hspace{5.4em} \leq
     t+\|\G\|^{1/p'}_{L^1(\om)}\Big(\int_\om|\G||u|^p\dx\Big)^{1/p}
      \\ \label{estim:nabla:G} & & \hspace{5.4em} \leq
       t+C\|u\|_{W^{1,p}(\om)}.
\eea
for some positive constant $C$ that depends only on $n,p,q,\om,\|\G\|_{M^q(p;\om')}$ and $\|\G\|_{L^1(\om)},$ where we have used Theorem \ref{theorem:uncertainty}-($ii$) in the last inequality. Thus, applying also assumption (E), we obtain
 \be\label{estim:coer:nabla}
  \|\nabla u\|^p_{L^p(\om;\Rn)}\leq c_1+c_2\|u\|_{W^{1,p}(\om)},
 \ee
where $c_1,c_2$ are positive constants independent of $u.$ Next observe that $u-f\in\W(\om),$ so that
 \bea\nonumber &  &
  \|u\|_{L^p(\om)}
   \leq
    \|u-f\|_{L^p(\om)} + \|f\|_{L^p(\om)}
     \\ \nonumber & & \hspace{4em} \leq
      C_P\|\nabla(u-f)\|_{L^p(\om;\Rn)} + \|f\|_{L^p(\om)},
 \eea
for a positive constant $C_P$ depending only on $n$ and $\om$, because of the Poincar\'{e} inequality in $\W(\om)$. Using (E) we have successively
 \bea\nonumber &  &
  \|u\|_{L^p(\om)}
   \leq
    C_P\big(\|\nabla u\|_{L^p(\om;\Rn)}+\|\nabla f\|_{L^p(\om;\Rn)}\big) + \|f\|_{L^p(\om)}
     \\ \nonumber & & \hspace{4em} \leq
      \frac{C_P}{\theta_\om}\Bigg(\Big(\int_\om|\nabla u|_A^p\dx\Big)^{1/p}+\|\nabla f\|_{L^p(\om;\Rn)}\Bigg) + \|f\|_{L^p(\om)}
       \\ \nonumber & & \hspace{4em} \leq
        \frac{C_P}{\theta_\om}\Big(\big(t+C\|u\|_{W^{1,p}(\om)}\big)^{1/p}+\|\nabla f\|_{L^p(\om;\Rn)}\Big) + \|f\|_{L^p(\om)},
 \eea
with $C$ as in (\ref{estim:nabla:G}). This implies the estimate
 \be\label{estim:coer:u}
  \|u\|^p_{L^p(\om)}\leq c_3+c_4\|u\|_{W^{1,p}(\om)},
 \ee
where $c_3,c_4$ are positive constants independent of $u.$ Now (\ref{estim:coer:nabla}) and (\ref{estim:coer:u}) give
 \[
  \|u\|^p_{W^{1,p}(\om)}
   \leq
    c_5+c_6\|u\|_{W^{1,p}(\om)},
 \]
for some positive constants $c_5,c_6$ that are independent of $u$. This implies in turn $\|u\|_{W^{1,p}(\om)}\leq\max\{1,(c_5+c_6)^{1/(p-1)}\}$, and (\ref{estim:W:crcv}) is proved.

(b) Let us prove the coercivity of $J$ in $\W(\om)$. Assume that $J[u]\leq t$ in (\ref{subcriticality}), then by applying H\"{o}lder's inequality, we obtain
 \bea\nonumber & &
  \e\|u\|_{L^p(\om)}^p
   \leq
    t+\int_\om\G u\dx
     \\ \nonumber & & \hspace{4.5em} \leq
      t+\|\G\|_{L^{p'}(\om)}\|u\|_{L^p(\om)}.
 \eea
This implies the estimate
 \be\label{u:crcv}
  \|u\|_{L^p(\om)}
   \leq m:=
    \max\Bigg\{1,\left(\frac{t+\|\G\|_{L^{p'}(\om)}}{\e}\right)^{1/(p-1)}\Bigg\}.
 \ee
From $J[u]\leq t$, applying once more H\"{o}lder's inequality and also the Morrey-Adams theorem (Theorem~\ref{theorem:uncertainty}-($i$)) we get
 \bea\nonumber &  &
  \int_\om|\nabla u|_A^p\dx
   \leq
    t+\int_\om\G u\dx+\int_\om|\V||u|^p\dx
     \\ \label{nablau:crcv} & & \hspace{5.5em} \leq
      t
       +\|\G\|_{L^{p'}(\om)}\|u\|_{L^p(\om)}
        +\delta\|\nabla u\|_{L^p(\om;\Rn)}^p
         +C'\|u\|^{p}_{L^{p}(\om)},
 \eea
where $C'=C_{n,p,q}\delta^{-n/(pq-n)}\|\V\|_{M^q(p;\om)}^{pq/(pq-n)}$. Thus, from (\ref{u:crcv}), (\ref{nablau:crcv}) and assumption (\rm E) we have for $\delta<\theta_\om^p$,
 \be\nonumber
  (\theta_\om^p-\delta)\|\nabla u\|_{L^p(\om;\Rn)}^p
   \leq
    t+\|\G\|_{L^{p'}(\om)}m+C'm^p,
 \ee
which, together with (\ref{u:crcv}), implies $\|u\|_{W^{1,p}(\om)} \leq C$.
\end{proof}

\begin{remark}
Propositions~\ref{proposition:wlsc} and \ref{proposition:coercivity} will be used to prove the existence of a minimizer for the  Rayleigh-Ritz variational problem (\ref{first:eigenvalue}), and to establish the weak comparison principle using the sub/supersolution method (see \S5.1).
\end{remark}
\subsection{Existence, properties and characterization of the positivity of $\lambda_1$}\label{subsec:eigenvalue}
The following theorem generalizes several results in the literature concerning the principal eigenvalue $\lambda_1$ (see for example \cite[Theorem 2.1]{AllH}, \cite[Proposition 2]{An}, \cite[Lemma 3]{GMdL}, \cite[Lemma 6.4]{PR}). Note that our results applies to a general bounded domain, and in particular, the boundary point lemmas are not used in the proof (cf.  \cite[Lemma 3]{GMdL} and \cite{PR}). In addition, we do not need any further regularity assumption on the entries of the matrix $A$ as in the aforementioned references, while the potential $V$ is far from being bounded.

\begin{theorem}\label{proposition:principal}
Let $\om$ be a bounded domain in $\Rn$, and assume that $A$ is a uniformly elliptic, bounded matrix in $\om$, and $V\in M^q(p;\om)$. Then the operator $Q'_{A,p,V}$ in $\om$ admits a principal eigenvalue $\lambda_1$ given by the Rayleigh-Ritz variational formula (\ref{first:eigenvalue}). Moreover, $\lambda_1$ is the only principal eigenvalue, it is simple and an isolated eigenvalue in $\R$.
\end{theorem}

\begin{proof} We define $\lambda_1$ by (\ref{first:eigenvalue}) and prove that it is a principal eigenvalue. Using the Morrey-Adams theorem (Theorem~\ref{theorem:uncertainty}) with $\delta=\theta_\om^p$ one sees that
 \be\nonumber
  \lambda_1
   \geq
     -C(n,p,q)\theta_\om^{-np/(pq-n)}\|V\|_{M^q(p;\om)}^{pq/(pq-n)}
     >-\infty.
 \ee
In particular, setting $\V:=V-\lambda_1+\e$, with $\e>0$, we get that
 \be\nonumber
  Q_{A,p,\V}[u;\om]
   \geq
    \e\|u\|_{L^p(\om)}^p
     \qquad \mbox{for all }u\in\W(\om).
 \ee
Applying Propositions \ref{proposition:wlsc}-($a$) and \ref{proposition:coercivity}-($b$) with $\G\equiv0$, we get that $Q_{A,p,V-\lambda_1+\e}[\cdot;\om]$ is coercive and weakly lower semicontinuous in $\W(\om)$, and consequently, also in $\W(\om)\cap\{\|u\|_{L^p(\om)}=1\}$. Hence, the infimum
 \be\nonumber
  \e=\inf_{u\in \W(\om)\setminus\{0\}}\frac{Q_{A,p,V-\lambda_1+\e}[u;\om]}{\|u\|^p_{L^p(\om)}},
 \ee
is attained in $\W(\om)\setminus\{0\}$ (see e.g., \cite[Theorem~1.2]{Str}), and thus $\lambda_1$ is attained in $\W(\om)\setminus\{0\}$.

Let $v_1$ be a minimizer of (\ref{first:eigenvalue}). It is quite standard to see that $v_1$ is a solution of (\ref{eigenproblem}) with $\lambda=\lambda_1$. Since  $|v_1|\in\W(\om)\setminus\{0\}$, it follows that $\big|\nabla(|v_1|)\big|_A=|\nabla v_1|_A$ a.e. in $\om$. This implies that $|v_1|$ is also a minimizer of (\ref{first:eigenvalue}) and thus a nonnegative solution of (\ref{eigenproblem}) with $\lambda=\lambda_1$. By the Harnack inequality, and the H\"older continuity of $|v_1|$,  we obtain that $|v_1|$ is strictly positive in $\om$. In light of the homogeneity of the eigenvalue problem (\ref{eigenproblem}), we may assume that $v_1$ is strictly positive in $\om$.

\medskip

\noindent To prove the simplicity of $\lambda_1$, we assume  that $v_2\in\W(\om)$ is another eigenfunction  of (\ref{eigenproblem}) with $\lambda=\lambda_1$. Hence, $v_2$ is a minimizer of (\ref{first:eigenvalue}), and thus has a definite sign. Without loss of generality, we may assume that $v_2>0$ in $\om$.  Applying Lemma~\ref{lemma:lind}-(ii) with $V_1=V_2=V$, $\lambda=\mu=\lambda_1$ and $w_\lambda=v_1,~w_\mu=v_2$ we obtain
 \be\nonumber
  0
   \geq
    c_p\Bigg\{
       \begin{array}{ll}
        \int_\om(v^p_1+v^p_2)\big|\nabla\log\frac{v_1}{v_2}\big|_A^p\dx, & \mbox{if }p\geq2,
         \\[8pt]
        \int_\om(v^p_1+v^p_2)\big|\nabla\log\frac{v_1}{v_2}\big|_A^2\big(|\nabla\log v_1|_A+|\nabla\log v_2|_A\big)^{p-2}\dx, & \mbox{if }p<2,
       \end{array}
 \ee
from which because of (\rm{E}) we deduce $|v_2\nabla v_1-v_1\nabla v_2|=0$ a.e. in $\om$, which in turn implies the existence of a positive constant $c$ such that $v_2=cv_1$ a.e. in $\om$.

\medskip

\noindent Next we show that $\lambda_1$ is the only eigenvalue possessing a nonnegative eigenfunction associated to it. If $\lambda>\lambda_1$ is an eigenvalue with eigenfunction
$\e v_\lambda\geq 0$, where $\e>0$ is small. Then by Lemma~\ref{lemma:lind}-(ii) with $V_1=V_2=V$, $\mu=\lambda_1$, and $w_\mu=v_1$, we have
 \be\nonumber
  (\lambda-\lambda_1)\int_\om(\e v^p_\lambda-v^p_1)\dx
   \geq 0,
 \ee
which is a contradiction for $\e$ small enough.

\medskip

\noindent It remains thus to prove that $\lambda_1$ is an isolated eigenvalue in $\R$. Suppose that there exists a sequence of eigenvalues $\{\lambda_k\}_{k\in\N}\subset \R$ such that $\lambda_k\downarrow\lambda_1$, as $k\to\infty$. Let $\{v_k\}_{k\in\N}$ be a sequence of the associated normalized eigenfunctions. We claim that $\{v_k\}_{k\in\N}$ is bounded in $\W(\om)$. Indeed, by the Morrey-Adams theorem, we obtain for some $0<\delta<1$ that
\bea\nonumber & &
  \int_\om|\nabla v_k|_A^p\dx
   \leq
    |\lambda_k| + \int_\om|V||v_k|^p\dx
     \\ \label{unif:isolation} & & \hspace{6em} \leq
      \delta\|\nabla v_k\|_{L^p(\om;\Rn)}^p+C,
 \eea
which implies our claim. Therefore, up to a subsequence, $v_k$ convergence weakly in $\W(\om)$, and also in $L^p(\om)$.

Next we claim that $v_k\to w$ in $\W(\om)$. Since $v_k\rightharpoonup w$ in $\W(\om)$, it is enough to show that $\{\|\nabla v_k\|_{L^p(\om;\Rn)}\}$ is a Cauchy sequence. Let $\e>0$ be arbitrary. The inequality
\[|a^p-b^p|\leq p|a-b|(a^{p-1}+b^{p-1})\qquad a,b\geq0,\] together with the H\"older inequality, and the Morrey-Adams theorem imply for all sufficiently large $k,l\in\N$
\bea\nonumber
    \Big|\int_\om|\nabla v_k|_A^p\dx-\int_\om|\nabla v_l|_A^p\dx\Big|
     &  &
      \\ \nonumber & & \hspace{-12em}\leq
       |\lambda_k-\lambda_l| + \int_\om|V|\big||v_k|^p- |v_l|^p\big|\dx
        \\ \nonumber & & \hspace{-12em}\leq
         \varepsilon + p\int_\om|V||v_k-v_l|\big||v_k|^{p-1}+ |v_l|^{p-1}\big|\dx.
          \\ \label{aa} & & \hspace{-12em}\leq
           \e + C(p)\Big(\int_\om|V||v_k-v_l|^p\dx\Big)^{1/p}\Big(\int_\om|V||v_k|^p\dx+\int_\om|V||v_l|^p\dx\Big)^{1/p'}.
\eea
Applying first the Morrey-Adams theorem and then (\ref{unif:isolation}), we see that both integrals on the second factor of (\ref{aa}) are uniformly bounded in $k,l$ respectively. For the first factor we use again the Morrey-Adams theorem to arrive at
\bea\nonumber
 & & \Big|\int_\om|\nabla v_k|_A^p\dx-\int_\om|\nabla v_l|_A^p\dx\Big|
  \\ \label{cauchy} & & \hspace{4em} \leq
   \e + C_1\Big(\e\int_\om|\nabla(v_k-v_l)|^p\dx +C_2\e^{n/(n-pq)}\int_\om|v_k-v_l|^p\dx\Big)^{1/p},
\eea
where $C_1,C_2$ are positive constants independent of $k,l$. The convergence in $L^p(\om)$ of $v_k$ to $v$ implies that there exists $m_\e\in\N$ such that
\[
 \int_\om|v_k-v_l|^p\dx\leq\e^{n/(pq-n)+1}\qquad\mbox{for all }k,l\geq m_\e.
\]
Coupling this with (\ref{cauchy}) implies that $\{\|\nabla v_k\|_{L^p(\om;\Rn)}\}$ is a Cauchy sequence.

By a similar argument, one shows that
$$Q_{A,p,V}[w]=\lambda_1\|w\|^p_{L^p(\om)}\,, $$
hence,  $w$ is a minimizer of the Rayleigh-Ritz variational problem \eqref{first:eigenvalue}, and hence an eigenfunction of (\ref{eigenproblem}) with $\lambda=\lambda_1$. The simplicity of $\lambda_1$ implies that $w=\pm v$, where $v>0$ is the normalized principal eigenfunction with an eigenvalue $\lambda_1$. Without loss of generality, we may assume that $v_k\rightarrow v$ in $\W(\om)$.

Set $\om_k^-:=\{x\in\om~|~v_k<0\}$. By Lemma \ref{lemma:Kato} (with $\V=V-\lambda_k$) we have that $v_k^-$ is a subsolution of $Q'_{A,p,V-\lambda_k}[u]=0$ in $\om,$ and thus from (\ref{eigendef})
 \bea\nonumber & &
  \int_\om|\nabla v_k^-|_A^p\dx
   \leq
    \int_\om|V-\lambda_k||v_k^-|^p\dx
     \\ \nonumber & & \hspace{6em} \leq
      \delta\|\nabla v_k^-\|^p_{L^p(\om;\Rn)}+C(n,p,q)\delta^{-n/(pq-n)}\|V-\lambda_k\|_{M^q(p;\om)}^{pq/(pq-n)}\|v_k^-\|^p_{L^p(\om)},
 \eea
for any $\delta>0$, where we have used Theorem~\ref{theorem:uncertainty}. For $\delta<\theta_\om^p$ we deduce because of assumption (\rm{E}) that
 \be\nonumber
  (\theta_\om^p-\delta)\|\nabla v_k^-\|^p_{L^p(\om;\Rn)}
   \leq
    C(n,p,q)\delta^{-n/(pq-n)}\|V-\lambda_k\|_{M^q(p;\om)}^{pq/(pq-n)}\|v_k^-\|^p_{L^p(\om)}.
 \ee
Since $v_k^-\equiv0$ in $\om\setminus\om_k^-$, we use Poincar\'{e}'s inequality
 \be\label{poincare}
  \|v_k^-\|_{L^p(\om)}
   \leq
    \Big(\frac{\Ln(\om_k^-)}{\Ln(B_1)}\Big)^{1/n}
     \|\nabla v_k^-\|_{L^p(\om;\Rn)},
 \ee
to get
 \be\nonumber
  (\theta_\om^p-\delta)\|\nabla v_k^-\|^p_{L^p(\om;\Rn)}
   \leq
    C(n,p,q)\delta^{-n/(pq-n)}\|V-\lambda_k\|_{M^q(p;\om)}^{pq/(pq-n)}\Big(\frac{\Ln\big(\om^-_k\big)}{\Ln(B_1)}\Big)^{p/n}\|\nabla v_k^-\|^p_{L^p(\om;\Rn)}.
 \ee
Canceling $\|\nabla v_k^-\|^p_{L^p(\om;\Rn)}$, rearranging and raising to the $n/p$ we arrive at
 \be\label{estim:omega:minus}
  \Ln\big(\om_k^-\big)
   \geq
    C(n,p,q)\Ln(B_1)(\theta_\om^p-\delta)^{n/p}\delta^{n^2/[p(pq-n)]}\|V-\lambda_k\|_{M^q(p;\om)}^{-nq/(pq-n)}.
 \ee
Notice that $\|V-\lambda_1\|_{M^q(p;\om)}$ is a strictly positive number. Indeed, assume that $\|V-\lambda_1\|_{M^q(p;\om)}=0$. Then $v_1$ is a nontrivial solution of the Dirichlet problem for the $(p,A)$-Laplace operator which is false under our assumptions on $A$ (see for example \cite{HKM,PS}).

On the other hand, $\|V-\lambda_k\|_{M^q(p;\om)}\to \|V-\lambda_1\|_{M^q(p;\om)}$ as $k\to\infty$. Therefore, there exists $C>0$ such that
 \be\label{estim:omega:morrey}
  \|V-\lambda_k\|_{M^q(p;\om)}
   \geq
    C\|V-\lambda_1\|_{M^q(p;\om)}\qquad \forall k\geq k_0.
 \ee
Consequently, (\ref{estim:omega:morrey}) applied to (\ref{estim:omega:minus}) implies that
$$\Ln\big(\om_k^-\big)\geq C>0 \qquad \forall k\geq k_0,$$
for a positive constant $C$ independent on $k$.

With this at hand, the rest of the proof follows \cite[Th\'{e}or\`{e}me 2]{An}. We include it for completeness: Let $\eta>0$. Recalling that $v$ is continuous in $\om$, we may pick a compact set $\om_\eta\Subset\om$ and $m_\eta>0$, such that $\Ln(\om\setminus\om_\eta)<\eta$ and $v(x)\geq m_\eta$ for every $x\in\om_\eta$. Up to subsequence that we don't rename, $v_k$ converges to $v$ a.e. in $\om$, and thus in $\om_\eta$. By the Egoroff theorem (see \cite[\S1.2]{EvG}) we have the existence of a measurable set $\om'\subset\om_\eta$ with $\Ln(\om')<\eta$ such that $v_k$ converges uniformly to $v$ on $\om_\eta\setminus\om'$. Since $v\geq m_\eta>0$ in $\om_\eta$ we deduce that for any $k$ large enough we have $v_k\geq0$ on $\om_\eta\setminus\om'$. Thus, $\om_k^-\subset\om'\cup(\om\setminus\om_\eta)$, which implies that $\Ln\big(\om_k^-\big)\leq2\eta$. Since $\eta>0$ is arbitrary, for $k$ large enough this contradicts our estimate $\Ln\big(\om_k^-\big)\geq C_1$. \end{proof}
We are now ready to prove the main result of this section. Extending the corresponding results in \cite{GMdL,PR}. We have
\begin{theorem}\label{theorem:GL}
Let $\om$ be a bounded domain,  and assume that $A$ is a uniformly elliptic, bounded matrix in $\om$, and $V\in M^q(p;\om)$. Consider the following assertions:
 \bea\nonumber
 &\alpha_1:&  Q'_{A,p.V}\mbox{ {\it satisfies the weak maximum principle in $\om$}}.
   \\ \nonumber
 &\alpha_2:& Q'_{A,p.V} \mbox{ {\it  satisfies the strong maximum principle in $\om$}}.
   \\ \nonumber
 &\alpha_3:& \lambda_1>0.
   \\ \nonumber
 &\alpha_4:& \mbox{{\it The equation $Q'_{A,p,V}[v]=0$ admits a positive strict supersolution in $\W(\om)$}}.
   \\ \nonumber
 &\alpha'_{4}:& \mbox{{\it The equation $Q'_{A,p,V}[v]=0$ admits a positive strict supersolution in  $W^{1,p}(\om)$}}.
   \\ \nonumber
 &\alpha_5:& \mbox{{\it For $0\leq g\in L^{p'}(\om)$, there exists a unique nonnegative solution in $\W(\om)$ of $Q'_{A,p,V}[v]=g$}}.
 \eea
Then $\alpha_1\Leftrightarrow\alpha_2\Leftrightarrow\alpha_3\Rightarrow\alpha_4\Rightarrow\alpha'_4$, and $\alpha_3\Rightarrow\alpha_5\Rightarrow\alpha_4$.
\end{theorem}

\begin{remark}\label{remark:theorem:GL}
In Corollary~\ref{cor_equiv} we prove (imposing stronger regularity assumptions on $A$ and $V$ when $p<2$) that in fact, $\alpha'_4\Rightarrow\alpha_3$. Hence, under these additional assumptions for $p<2$, all the above assertions are equivalent.
\end{remark}

\begin{proof} $\alpha_1\Rightarrow\alpha_2$. Let $v\in W^{1,p}(\om)$ be a solution of (\ref{Q'=f:om}) and suppose $v\geq0$ on $\partial\om$.  The nonnegativity of $g$ and the weak maximum principle implies that $v$ is a nonnegative supersolution of (\ref{Q'=0}) in $\om$. Suppose that for some $x_0,x_1\in\om$ we have $v(x_0)\neq0$ and $v(x_1)=0$ and let $\om'\Subset\om$ contain both $x_0$ and $x_1$. Recalling Remark \ref{remark:harnack_supersol}, we apply the weak Harnack inequality if $p\leq n$, or the Harnack inequality if $p>n$, to get $v\equiv0$ in $\om'$. This contradicts the assumption that $v(x_0)\neq0$. Thus, if $v\neq 0$ we necessarily have $v>0$ in $\om$.

\medskip

\noindent$\alpha_2\Rightarrow\alpha_3$. Suppose that  $\lambda_1\leq0$ and let $v\in\W(\om)$ be the corresponding principal eigenfunction. Then $u:=-v$ is a supersolution of the equation (\ref{Q'=0}) in $\om$, satisfying $u=0$ on $\partial\om$, and $u\neq 0$. By the strong maximum principle, $u$ is positive which is absurd.

\medskip

\noindent$\alpha_3\Rightarrow\alpha_1$. Let $v\in W^{1,p}(\om)$ be a solution of (\ref{Q'=f:om}) such that $v\geq0$ on $\partial\om$. Taking $v^-\in\W(\om)$ as a test function we see that
 \be\nonumber
  Q_{A,p,V}[v^-;\om]
   =
    \int_{\om^-}gv\dx,
 \ee
where $\om^-:=\{x\in\om~|~v<0\}$. The nonnegativity of $g$ gives $Q_{A,p,V}[v^-;\om]\leq0$, which implies that $\lambda_1\leq0$. Thus, we must have $v^-=0$ a.e. in $\om$, or in other words $v\geq0$ a.e. in $\om$.

\medskip

\noindent$\alpha_3\Rightarrow\alpha_4$. Since $\lambda_1>0$, it follows that the principal eigenfunction is a positive strict supersolution of the equation (\ref{Q'=0}) in $\om$.

\medskip

\noindent$\alpha_4\Rightarrow\alpha'_4$. This is trivial.

\medskip

\noindent$\alpha_3\Rightarrow\alpha_5$.  Consider the functional
$$J[u] := Q_{A,p,V}[u;\om]-\int_\om gu\dx \qquad u\in\W(\om).$$
By Proposition~\ref{proposition:wlsc}-($a$), $J$ is weakly lower semicontinuous in $\W(\om)$, and by Proposition~\ref{proposition:coercivity}-($b$), $J$ is coercive. Therefore, the corresponding Dirichlet problem admits a solution  $v_1\in\W(\om)$ (see for example, \cite[Theorem 1.2]{Str}). Since  $\alpha_3\Rightarrow\alpha_2$, this solution is either zero or strictly positive.

If $v_1=0$, then $g=0$, and by the uniqueness of the principal eigenvalue, equation (\ref{Q'=0}) in $\om$ does not admit a positive solution in $\W(\om)$. So, we may assume that $v_1>0$ and let $v_2\in\W(\om)$ be another positive solution. Applying Lemma~\ref{lemma:lind}-(i) with $g_1=g_2=g$ and $V_1=V_2=V$, we obtain
 \be\nonumber0\geq
  \int_\om g\Big(\frac{1}{{v^{p-1}_{1,h}}}-\frac{1}{{v^{p-1}_{2,h}}}\Big)(v^p_{1,h}-v^p_{2,h})\dx
   \geq
    \int_\om V\Big[\Big(\frac{v_1}{{v_{1,h}}}\Big)^{p-1}-\Big(\frac{v_2}{{v_{2,h}}}\Big)^{p-1}\Big](v^p_{1,h}-v^p_{2,h})\dx.
  \ee
The integrand of the integral on the right converges to $0$ a.e. in $\om$, and also it satisfies the following estimate for every $h<1$
 \be\nonumber
  \Big|V\Big[\Big(\frac{v_1}{v_{1,h}}\Big)^{p-1}-\Big(\frac{v_2}{v_{2,h}}\Big)^{p-1}\Big](v^p_{1,h}-v^p_{2,h})\Big|
   \leq
    2|V|[(v_1+1)^p+(v_2+1)^p]\in L^1(\om).
 \ee
Thus
 \be\nonumber
  \lim_{h\to0}\int_\om g\Big(\frac{1}{{v^{p-1}_{1,h}}}-\frac{1}{{v^{p-1}_{2,h}}}\Big)(v^p_{1,h}-v^p_{2,h})\dx
   =
    0,
 \ee
which together with Fatou's lemma imply that the right hand side of \eqref{Ln} equals zero. Thus,  $v_2=v_1$ a.e. in $\om$.

\medskip

\noindent$\alpha_5\Rightarrow\alpha_4$. Let $v\in\W(\om)$ be a positive solution of (\ref{Q'=f:om}) with $g\equiv1$. Then $v$ is readily a positive strict supersolution of (\ref{Q'=0}) in $\om$.
\end{proof}
\section{Positive global solutions}\label{sec:positivity}
In the present section we pass from local to global properties of positive solutions of the equation (\ref{Q'=0}) in $\Om$. In \S\ref{subsec:aap} we establish the AP theorem, while in \S\ref{subsec:criticality} we prove among other results the equivalence of the first four statements of the Main Theorem.
\subsection{The AP theorem}\label{subsec:aap}
In this subsection we prove the AP theorem for the operator $Q'_{A,p,V}$ under hypothesis (H0). We will add a couple of equivalent assertions to this theorem, regarding the following first-order equation
 \be\label{-divT+f(T)=V}
 -\Div_A T+(p-1)|T|_A^{p'}=V\qquad \mbox{in }\Om,
 \ee
where $\Div_AT = \Div(AT)$ and $T\in L^{p'}_{\rm loc}(\Om;\Rn)$; see \cite[Theorem 1.3]{JMV1} for a similar study when $A=I_n$, and $p=2$.
\begin{definition}\label{def.weak:sol:-divT+f(T)=V}
 Suppose the matrix $A$ satisfies ({\rm S}),~({\rm E}) and let $V\in L^1_{\rm loc}(\Om)$. A vector field $T\in L^{p'}_{\rm loc}(\Om;\Rn)$ is a {\em solution} of (\ref{-divT+f(T)=V}) in $\Om$ if
 \be\label{weak:sol:div}
  \int_\Om AT\cdot\nabla u\dx
   +(p-1)\int_\Om |T|_A^{p'}u\dx
    =\int_\Om Vu\dx
     \qquad \mbox{for all }u\in\test(\Om),
 \ee
and a {\em (super)subsolution} of (\ref{-divT+f(T)=V}) in $\Om$ if
 \be\label{weak:ssol:div}
  \int_\Om AT\cdot\nabla u\dx
   +(p-1)\int_\Om |T|_A^{p'}u\dx
    \;(\geq)\leq\int_\Om Vu\dx
     \qquad \mbox{for all nonnegative }u\in\test(\Om).
 \ee
\end{definition}

\begin{remark}
The additional assumption $V\in M^q_{\rm loc}(p;\Om)$ allows the replacement of $\test(\Om)$ in Definition~\ref{def.weak:sol:-divT+f(T)=V} by $\wtest(\Om)$.
\end{remark}

\begin{theorem}[The AP theorem]\label{theorem:AP} Under hypothesis (H0), the following assertions are equivalent:
 \bea\nonumber
 &\mathcal{A}_1:& Q_{A,p,V}[u]\geq0\mbox{ for all }u\in\test(\Om).
   \\ \nonumber
 &\mathcal{A}_2:& \mbox{$Q'_{A,p,V}[w]=0$ admits a positive solution }v\in W_{\rm loc}^{1,p}(\Om).
   \\ \nonumber
 &\mathcal{A}_3:& \mbox{$Q'_{A,p,V}[w]=0$ admits a positive supersolution }\tilde{v}\in W_{\rm loc}^{1,p}(\Om).
   \\ \nonumber
 &\mathcal{A}_4:& \mbox{(\ref{-divT+f(T)=V}) admits a solution }T\in L^{p'}_{\rm loc}(\Om;\Rn).
   \\ \nonumber
 &\mathcal{A}_5:& \mbox{(\ref{-divT+f(T)=V}) admits a subsolution }\tilde{T}\in L^{p'}_{\rm loc}(\Om;\Rn).
 \eea
\end{theorem}

\begin{proof} We prove $\mathcal{A}_1\Rightarrow\mathcal{A}_2\Rightarrow\mathcal{A}_j\Rightarrow\mathcal{A}_5\Rightarrow\mathcal{A}_1$, where $j=3,4$.

\medskip

\noindent$\mathcal{A}_1\Rightarrow\mathcal{A}_2$. We fix a point $x_0\in\Om$ and let $\{\om_i\}_{i\in\N}$ be a sequence of Lipschitz domains such that $x_0\in\om_1$, $\om_i\Subset\om_{i+1}\Subset\Om$,  $i\in\N$, and $\cup_{i\in\N}\om_i=\Om$. For  $i\geq 2$, we consider the problem
 \be\label{ap:Q'=f_i}
\left\{
  \begin{array}{ll}
    Q'_{A,p,V+1/i}[u]=f_i & \mbox{in } \om_i, \\[2mm]
     u=0 & \mbox{on }\partial\om_i,
  \end{array}
\right.
 \ee
where $f_i\in\test(\om_i\setminus\overline{\om}_{i-1})\setminus\{0\}$ are nonnegative. Assertion $\mathcal{A}_1$ implies
 \be\nonumber
  \lambda_1(Q_{A,p,V+1/i};\om_i)\geq\frac{1}{i}\qquad \mbox{for all }i\in\N,
 \ee
so that by Theorem~\ref{theorem:GL} there exists a positive solution $v_i\in\W(\om_i)$ of (\ref{ap:Q'=f_i}). Since $\sprt\{f_i\}\subset\om_i\setminus\overline{\om}_{i-1}$, setting $\om_i'=\om_{i-1}$, we have
 \be\label{ap:Q'=0_i}
  \int_{\om_i}|\nabla v_i|^{p-2}_AA\nabla v_i\cdot\nabla u\dx+\int_{\om_i}(V+1/i)v_i^{p-1}u\dx=0
   \qquad \mbox{for all }u\in\W(\om_i').
 \ee
By Theorem~\ref{thrm:regularity}, the solutions $v_i$ we have obtained are continuous. We may thus normalize $f_i$ so that $v_i(x_0)=1$ for all $i\in\N$. To arrive to the desired conclusion we apply the Harnack convergence principle (Proposition \ref{Harnack:cp}) with $\V_i:=V+1/i.$

\medskip

\noindent$\mathcal{A}_2\Rightarrow\mathcal{A}_3$. This is immediate with $\tilde{v}=v$.

\medskip

\noindent$\mathcal{A}_2\Rightarrow\mathcal{A}_4$ and  $\mathcal{A}_3\Rightarrow\mathcal{A}_5$. Let $v$ be a positive (super)solution of (\ref{Q'=0}). By the weak Harnack inequality (Remark \ref{remark:harnack_supersol}) in case $p\leq n$, or by the Harnack inequality if $p>n$, we have $1/v\in L^\infty_{\rm loc}(\Om)$. Set

 \be\nonumber
  T
   :=
    -|\nabla\log v|_A^{p-2}\nabla\log v,
 \ee
and let $u\in\test(\Om)$. We may thus pick $|u|^pv^{1-p}\in\wtest(\Om)$ as a test function in (\ref{weak:ssol:Q}) to get
 \be\label{test1}
  (p-1)\int_\Om|T|_A^{p'}|u|^p\dx
   \leq
    p\int_\Om|T|_A|u|^{p-1}|\nabla u|_A\dx
     +
      \int_\Om V|u|^p\dx,
 \ee
Note that from (\ref{test1}) we obtain $\mathcal{A}_1$ just by using Young's inequality $pab\leq(p-1)a^{p'}+b^p$ with $a=|T|_A|u|^{p-1}$ and $b=|\nabla u|_A$ in the first term of the right hand side. Towards $\mathcal{A}_3$, we use instead Young's inequality
 \be\label{epsilon:Young}
  p a b
   \leq
    \eta a^{p'}+\Big(\frac{p-1}{\eta}\Big)^{p-1}b^p,
 \ee
with $\eta\in(0,p-1)$ and the above $a,b$. We arrive at
 \bea\nonumber
  (p-1-\eta)\int_\Om|T|_A^{p'}|u|^p\dx
   \leq
    \Big(\frac{p-1}{\eta}\Big)^{p-1}\int_\Om|\nabla u|_A^p\dx
     +
      \int_\Om|V||u|^p\dx.
 \eea
This, together with $({\rm E})$ and Theorem~\ref{theorem:uncertainty} imply by specializing $u$ that $T\in L^{p'}_{\rm loc}(\Om;\Rn)$. Next we show that $T$ is a (sub)solution of (\ref{-divT+f(T)=V}). To this end, for $u\in\test(\Om)$, or for nonnegative $u\in\test(\Om)$, we pick $uv^{1-p}\in\wtest(\Om)$ as a test function in (\ref{weak:sol:Q}), or (\ref{weak:ssol:Q}) respectively, to obtain
 \be\nonumber
  -\int_\Om AT\cdot\nabla u\dx
   -(p-1)\int_\Om|T|_A^{p'}u\dx
    +\int_\Om Vu\dx
     \;(\geq)=0.
 \ee

\medskip

\noindent$\mathcal{A}_4\Rightarrow\mathcal{A}_5$. This is immediate with $\tilde{T}=T$.

\medskip

\noindent$\mathcal{A}_5\Rightarrow\mathcal{A}_1$. Suppose now that $T\in L^{p'}_{\rm loc}(\Om;\Rn)$ and let $u\in\test(\Om)$. We compute
 \bea\nonumber & &
  -\int_\Om AT\cdot\nabla(|u|^p)\dx
   =
    -p\int_\Om|u|^{p-1}AT\cdot\nabla|u|\dx
     \\ \nonumber & & \hspace{9.2em} \leq
      p\int_\Om|u|^{p-1}|T|_A|\nabla u|_A\dx
       \\ \nonumber & & \hspace{9.2em} \leq
        (p-1)\int_\Om|u|^p|T|_A^{p'}\dx+\int_\Om|\nabla u|_A^p\dx,
 \eea
where we have also used Young's inequality $pab\leq(p-1)a^{p'}+b^p$, with $a=|u|^{p-1}|T|_A$ and $b=|\nabla u|_A$. This readily implies
 \be\label{Harell}
  \int_\Om|\nabla u|_A^p\dx
   \geq
    -\int_\Om AT\cdot\nabla(|u|^p)\dx - (p-1)\int_\Om |T|_A^{p'}|u|^p\dx
     \qquad \mbox{ for all }u\in\test(\Om).
 \ee
If $T$ is a subsolution of (\ref{-divT+f(T)=V}), then testing (\ref{weak:ssol:div}) by $|u|^p$, one readily sees from (\ref{Harell}) that $Q_{A,p,V}[u]$ is nonnegative for any $u\in\test(\Om)$.
\end{proof}

\begin{remark} Inequality (\ref{Harell}) with $A=I_n$ has been obtained in \cite{FHdTh}.
\end{remark}
\subsection{Criticality theory}\label{subsec:criticality}
In the present subsection we generalize several global positivity properties of the functional $Q_{A,p,V},$ where $A$ and $V$ satisfy (at least) our regularity assumption (H0). For the convenience of the reader, we recall the following terminology.

\begin{definition}\label{def:critical}  Assume that $Q_{A,p,V}$ is {\em nonnegative} in $\Om$ (that is, $Q_{A,p,V}[u]\geq0$ for all $u\in\test(\Om)$) with coefficients satisfying hypothesis {\em(H0)}. Then $Q_{A,p,V}$ is called {\em subcritical} in $\Om$ if there exists a nonnegative weight function $W\in  M^q_{\rm loc}(p;\Om)\setminus\{0\}$ such that
 \be\label{ineq:subcriticality}
  Q_{A,p,V}[u]
   \geq
    \intom W|u|^p\dx\qquad \mbox{for all }u\in\test(\Om).
 \ee
If this is not the case, then $Q_{A,p,V}$ is called {\em critical} in $\Om$.

The functional $Q_{A,p,V}$ is called {\em supercritical} in $\Om$ if $Q_{A,p,V}$ is not nonnegative in $\Om$ (that is, there exists $u\in\test(\Om)$ such that $Q_{A,p,V}[u]<0$).
\end{definition}

\begin{definition}\label{def:nullseq} A sequence $\{u_k\}\subset\wdtest(\Om)$ is called a {\em null sequence} with respect to the nonnegative functional $Q_{A,p,V}$ in $\Om$ if
\vspace{.5em}

 a) $u_k\geq0$ for all $k\in\mathbb{N}$,
\vspace{.5em}

 b) there exists a fixed open set $K\Subset\Om$ such that $\|u_k\|_{L^p(K)}=1$ for all $k\in\mathbb{N}$,
\vspace{.5em}

 c) $\displaystyle{\lim_{k\to\infty}Q_{A,p,V}[u_k]=0}$.
\vspace{.5em}

\noindent We call a positive $\phi\in W_{\rm loc}^{1,p}(\Om)$ a {\em ground state} of $Q_{A,p,V}$ in $\Om$ if $\phi$ is an $L_{\rm loc}^p(\Om)$ limit of a null sequence.
\end{definition}

\begin{remark}\label{rem_pe_gs}
Let  $\om\subset\Rn$ be a bounded domain, and suppose that $A$ is uniformly elliptic and bounded matrix in $\om$, and $V\in M^q(p;\om)$. Let $v_1$ be the principal eigenfunction with eigenvalue $\lambda_1$.  Set $C_K:= \|v_1\|_{L^p(K)}$, where  $K\Subset\Om$ is fixed. Then the constant sequence $\{C_K^{-1}v_1\}$ is a null sequence and $C_K^{-1}v_1$ is a ground state of $Q_{A,p,V-\lambda_1}$ in $\om$.
\end{remark}

The following proposition states an elementary positivity property of the functional $Q_{A,p,V}$.

\begin{proposition}\label{proposition:superpotential}
Suppose that $V_2\geq V_1$ a.e. in $\Om$ and  $\Ln\big(\{V_2>V_1\}\big)>0$.
\vspace{.5em}

a) If $Q_{A,p,V_1}$ is nonnegative in $\Om$, then $Q_{A,p,V_2}$ is subcritical in $\Om$.
\vspace{.5em}

b) If $Q_{A,p,V_2}$ is critical in $\Om$, then $Q_{A,p,V_1}$ is supercritical in $\Om$.
\end{proposition}

\begin{proof} Part \textit{b)} follows from part \textit{a)} by contradiction, and from the obvious relation
\[
  Q_{A,p,V_2}[u]=Q_{A,p,V_1}[u]+\int_\Om(V_2-V_1)|u|^p\dx \qquad \mbox{for all }u\in\test(\Om),
\]
part \textit{a)} evident.
\end{proof}
Note here that definitions \ref{def:critical} and \ref{def:nullseq}, and also Proposition~\ref{proposition:superpotential} make perfect sense if $V$ is merely in $L^1_{\rm loc}(\Om)$ for all values of $p$.

Now we connect the criticality/subcriticality of the functional $Q_{A,p,V}$ in $\Om$ with the existence of positive weak (super)solutions problem for equation (\ref{Q'=0}) in $\Om$, through the existence of ground states. Towards this we need to give sufficient conditions on $A$ and $V$, under which a null sequence with respect to the nonnegative functional $Q_{A,p,V}$, will converge in $L^p_{\rm loc}$ to a function in $W_{\rm loc}^{1,p}$.

We need the following definition for the case $1<p<2$.
\begin{definition}
Suppose that $1<p<2$. A positive supersolution $v$ of (\ref{Q'=0}) will be called {\em regular} provided that $v$ and $|\nabla v|$ are locally bounded a.e. in $\Om$.
\end{definition}

\begin{remark} Under hypothesis (H1) for $1<p<2$, any positive supersolution $v$ of (\ref{Q'=0}) satisfying $Q_{A,p,V}[v]=g\geq 0 $ with $g\in L^{p'}_{\mathrm{loc}}(\Om)$ is regular (see Remark \ref{remark:local grad bound}).
\end{remark}

We start with the following proposition that gives us the intuition that any null sequence converges in some sense to {\em any} positive (regular if $p<2$) (super)solution. Note that our proof for the case $p<2$ is considerably shorter than the corresponding proof in \cite{PT1} and \cite{PR}.
\begin{proposition}\label{proposition:W1ploc:boundedness}
Suppose that $\{u_k\}\subset\wdtest(\Om)$ is a null sequence with respect to a nonnegative functional $Q_{A,p,V}$ in $\Om$ with coefficients satisfying hypothesis {\em(H0)}.

Let $v$ be a positive supersolution  of the equation (\ref{Q'=0}) in $\Om$. In case $1<p<2$ we assume further that $v$ is regular. Set $w_k:=u_k/v$. Then $\{w_k\}$ is bounded in $W_{\rm loc}^{1,p}(\Om)$, and  $\nabla w_k\rightarrow0$ in $L^p_{\rm loc}(\Om;\Rn)$.
\end{proposition}

\begin{proof} Let $K\Subset \Om$ be the set such that the null sequence $\{u_k\}$ satisfies $\|u_k\|_{L^p(K)}=1$ for all $k\in\mathbb{N}$. Fix a Lipschitz domain $\om$ such that  $K\Subset \om\Subset\Om$.

By Minkowski and Poincar\'{e} inequalities, and the weak Harnack inequality, we have
 \bea\nonumber & &
  \|w_k\|_{L^p(\om)}
   \leq
    \|w_k-\langle w_k\rangle_K\|_{L^p(\om)}+\langle w_k\rangle_K[\Ln(\om)]^{1/p}
     \\ \nonumber & & \hspace{4.6em} \leq
      C(n,p,\om,K)\|\nabla w_k\|_{L^p(\om;\Rn)}+\frac{1}{\inf_{K}v}\langle u_k\rangle_K[\Ln(\om)]^{1/p}.
 \eea
Since $\|u_k\|_{L^p(K)}=1$, applying Holder's inequality we deduce
 \be\label{Lpwkestimate}
  \|w_k\|_{L^p(\om)}
   \leq
    C(n,p,\om,K)\|\nabla w_k\|_{L^p(\om;\Rn)}+\frac{1}{\inf_{K}v}\Big[\frac{\Ln(\om)}{\Ln(K)}\Big]^{1/p}.
 \ee
Let
\[I(v,w_k):=C(p)\left\{
              \begin{array}{ll}
               \displaystyle{\int_\Om v^p|\nabla w_k|_A^p\dx} & p\geq 2,
               \\[6mm]
                \displaystyle{\int_\Om |\nabla w_k|_A^2\Big(|\nabla(vw_k)|_A+w_k|\nabla v|_A\Big)^{p-2}\!\!\!\!\dx} & 1\leq p <2,
              \end{array}
            \right.
\]
where $C(p)$ is the constant in \eqref{vecineqLindqvist}. We now use \eqref{vecineqLindqvist} with $\alpha=\nabla(w_kv)=\nabla u_k,~\beta=w_k\nabla v$ to obtain
\bea&  &
 I(v,w_k)
  \leq
   \intom|\nabla u_k|_A^p\dx-\intom w_k^p|\nabla v|_A^p\dx-\intom v|\nabla v|_A^{p-2}A\nabla v\cdot\nabla (w_k^p)\dx
    \\ \nonumber & & \hspace{3.8em} =
     \intom|\nabla u_k|_A^p\dx-\intom |\nabla v|_A^{p-2}A\nabla v\cdot\nabla(w_k^pv)\dx,
\eea
 Since $v$ is a positive supersolution, we get
 \be\label{common:geq2}
  I(v,w_k)
  \leq
   \intom|\nabla u_k|_A^p\dx
    +
     \intom Vu_k^p\dx = Q_{A,p,V}[u_k].
 \ee
Suppose now that $p\geq2$. Using the definition of $I$, and the weak Harnack inequality, we obtain from \eqref{common:geq2} that
 \be\label{nablawk0pgeq2}
  c\int_\om |\nabla w_k|^{p}\dx \leq C(p) \intom v^{p}|\nabla w_k|^{p}_{A}\dx
   \leq
    Q_{A,p,V}[u_k]
     \rightarrow 0\qquad \mbox{as }k\rightarrow\infty,
 \ee
where $c>0$ is a positive constant. By the weak compactness of $W^{1,p}(\om)$, we get for $p\geq2$ that (up to a subsequence)
 \be\label{LpwkGradestimate}
  \nabla w_k\rightarrow0\qquad \mbox{in }L^p_{\rm loc}(\Om;\Rn).
   \ee
By \eqref{Lpwkestimate} and \eqref{nablawk0pgeq2}, we have that $w_k$ is bounded in $W_{\rm loc}^{1,p}(\om)$ for any $p\geq2$.

\medskip

On the other hand if $p<2$, then by the definition of $I$ and \eqref{common:geq2}, we get
 \be\nonumber
  C(p)\intom\frac{v^2|\nabla w_k|_A^2}{\Big(|\nabla(vw_k)|_A+w_k|\nabla v|_A\Big)^{2-p}}\dx
   \leq
    Q_{A,p,V}[u_k]
     \rightarrow 0\qquad \mbox{as }k\rightarrow\infty.
 \ee
For convenience we set $q_k=Q_{A,p,V}[u_k]$. By H\"{o}lder's inequality with conjugate exponents $2/p$ and $2/(2-p)$, we get
 \bea\nonumber
  &  &\hspace{-2em}\int_\om v^p|\nabla w_k|_A^p\dx
   \\ \nonumber &  &\leq
    \Bigg(\intom\frac{v^2|\nabla w_k|_A^2}{\Big(|\nabla(vw_k)|_A+w_k|\nabla v|_A\Big)^{2-p}}\dx\Bigg)^{p/2}
     \Bigg(\int_\om\Big(|\nabla(vw_k)|_A+w_k|\nabla v|_A\Big)^p\dx\Bigg)^{1-p/2}
       \\ \nonumber &  &\leq
        C(p)^{-1}q_k^{p/2}
         \Bigg(\int_\om v^p|\nabla w_k|^p_A\dx+\int_\om w_k^p|\nabla v|^p_A\dx\Bigg)^{1-p/2}
\\ \nonumber &  &\leq
C(p)^{-1}q_k^{p/2}
         \Big(\int_\om v^p|\nabla w_k|^p_A\dx+\int_\om w_k^p|\nabla v|^p_A\dx + 1 \Big).
 \eea
Since $v$ is locally bounded, and locally bounded away from zero, and $|\nabla v|$ is locally bounded, and $A$ is uniformly elliptic and bounded in $\om$, we get using \eqref{Lpwkestimate} that for some positive constants  $c_j;~1\leq j\leq 4$, that are independent of $k$, there holds
 \bea\nonumber & &
  c_1\int_\om |\nabla w_k|^p\dx
   \leq
    c_2 q_k^{p/2}\Big(\int_\om |\nabla w_k|^p\dx+\int_\om w_k^p\dx+1\Big)
     \\ \nonumber & & \hspace{7em} \leq
      c_2 q_k^{p/2}\Big(c_3\int_\om |\nabla w_k|^p\dx+c_4\Big).
 \eea
Since $q_k\rightarrow0$ as $k\to\infty$,  we conclude that also in the case $p<2$ we have
  \be\nonumber
  \nabla w_k\rightarrow0\qquad \mbox{in }L^p_{\rm loc}(\Om;\Rn),
   \ee
and thus by (\ref{Lpwkestimate}) we have that $w_k$ is bounded in $W_{\rm loc}^{1,p}(\om)$ for any $p<2$.
\end{proof}

Several consequences follow. In the following statement, uniqueness is meant up to a positive multiplicative constant.
\begin{theorem}\label{crit}
Suppose that $Q_{A,p,V}$ is nonnegative in $\Om$ with $A$ and $V$ satisfying hypothesis {\em(H0)} if $p\geq2,$ or {\em (H1)} if $1<p<2$. Then any null sequence with respect to $Q_{A,p,V}$ converges, in $L^p_{\rm loc}$ and a.e. in $\Om$, to a unique positive (regular if $p<2$) supersolution of (\ref{Q'=0}) in $\Om$. In particular, a ground state is the unique positive solution and the unique positive (regular if $p<2$) supersolution of (\ref{Q'=0}) in $\Om$, and so the ground state is $C^\gamma$ if $p\geq2$, or $C^{1,\gamma}$ if $\,1<p<2$.
\end{theorem}

\begin{remark}  At this point we need to add the stronger assumption (H1) on $A$ and $V$ in the case $1<p<2$, since in this case we assume the existence of a positive regular (super)solution. In fact, the proof presented here for $p<2$ applies under the least assumptions on $A$ and $V$ that ensures the Lipschitz continuity of positive solutions. This fails if we just keep the assumption ${\rm (E)}$ on the matrix $A$, even for $V\equiv0$ (see \cite{JMVS}). To our knowledge, the least known assumptions on $A$ and $V$ ensuring the Lipschitz continuity of solutions are due to Lieberman \cite{Lb} (see our Remark \ref{remark:local grad bound}).
\end{remark}

\begin{proof}[Proof of Theorem~\ref{crit}] From the AP theorem we may fix a positive (regular if $p<2$) supersolution $v\in W^{1,p}_\mathrm{loc}(\Om)$ and a positive (regular if $p<2$) solution $\tilde{v}\in W^{1,p}_\mathrm{loc}(\Om)$ of (\ref{Q'=0}). Setting $w_k=u_k/v$ we have by Proposition~\ref{proposition:W1ploc:boundedness} that $\nabla w_k\rightarrow0$ in $L^p_{\rm loc}(\Om;\Rn)$. Rellich-Kondrachov theorem implies (see the proof of  \cite[Theorem 8.11]{LL}) that, up to a subsequence, $w_k\rightarrow c$ for some $c\geq0$ in $W^{1,p}_\mathrm{loc}(\Om)$. This implies in turn that, up to a further subsequence, $u_k\rightarrow cv$ a.e. in $\Om$, and also in $L^p_{\rm loc}(\Om)$. Consequently, $c=1/\|v\|_{L^p(K)}>0$. It follows that any null sequence $\{u_k\}$ converges (up to a positive multiplicative constant) to the same positive (regular if $p<2$) supersolution $v$. Since the solution $\tilde{v}$ is a (regular if $p<2$) supersolution, we see that $v=C\tilde{v}$ for some $C>0$, and therefore  it is also the unique positive solution of (\ref{Q'=0}) in $\Om$.
\end{proof}

We can now close the chain of implications between the assertions of Theorem~\ref{theorem:GL} (see Remark \ref{remark:theorem:GL}).

\begin{corollary}\label{cor_equiv}
Let $\om\Subset\Rn$ and suppose that $A$ is uniformly elliptic and bounded matrix in $\om$, and $V\in M^q(p;\om)$. In case $1<p<2$, we suppose in addition that $A$ and $V$ satisfy hypothesis {\em (H1)}.

If the equation $Q'_{A,p,V}[v]=0$ admits a positive, regular, strict supersolution in $W^{1,p}(\om)$, then the principal eigenvalue is strictly positive.

Hence, all assertions of Theorem~\ref{theorem:GL} are equivalent (if by a supersolution we mean, in case $p<2$, a regular one).
\end{corollary}

\begin{proof} $\alpha'_4\Rightarrow\alpha_3$. From the AP theorem we get $Q_{A,p,V}[u;\om]\geq0$ for all $u\in\test(\om)$, which implies that $\lambda_1\geq0$. Suppose that $\lambda_1=0$. Then by Remark~\ref{rem_pe_gs} and Theorem~\ref{crit}, the principal eigenfunction which is a positive (regular if $p<2$) solution of (\ref{Q'=0}) in $\om$ is the unique (regular if $p<2$) positive supersolution of that equation. This shows that this equation cannot have a positive strict (regular if $p<2$) supersolution.
\end{proof}

In the next theorem we state characterizations of criticality, subcriticality and existence of a null sequence. We also state a useful Poincar\'{e} inequality in the case where $Q_{A,p,V}$ is critical. It generalizes the corresponding results in \cite{PT0,PT1,PT2,PR,TT}.

\begin{theorem}\label{theorem:main} Suppose that $Q_{A,p,V}$ is nonnegative on $\test(\Om)$ with $A$ and $V$  satisfying hypothesis {\em(H0)} if $p\geq2,$ or {\em (H1)} if $1<p<2$. Then
\begin{itemize}
\item[(i)] $Q_{A,p,V}$ is critical in $\Om$ if and only if $Q_{A,p,V}$ admits a null sequence.
\item[(ii)] $Q_{A,p,V}$ admits a null sequence if and only if (\ref{Q'=0}) admits a unique positive (regular if $p<2$) supersolution.
\item[(iii)] $Q_{A,p,V}$ is subcritical in $\Om$ if and only if there exists a strictly positive weight function $W\in C^0(\Om)$ such that (\ref{ineq:subcriticality}) holds true.
\item[(iv)] If $Q_{A,p,V}$ admits a ground state $\phi$, then there exists a strictly positive weight function $W\in C^0(\Om)$ such that for every $\psi\in\test(\Om)$ with $\int_\Om\phi\psi dx\neq0$, the following Poincar\'{e} type inequality holds:
 \be\nonumber
  Q_{A,p,V}[u]+C\Big|\int_\Om u\psi\dx\Big|^p
   \geq
    \frac{1}{C}\int_\Om W|u|^p\dx \qquad\mbox{for all }u\in \W(\Om),
 \ee
and some positive constant $C>0$.
\end{itemize}
\end{theorem}

\begin{remark}
In the sequel (Lemma~\ref{lemma:unif:conv:gs}) we add the following accompanying to \emph{(i)} statement: {\it if $Q_{A,p,V}$ is critical in $\Om$, then there exists a null sequence that converges locally uniformly in $\Om$ to the ground state.}
\end{remark}

\begin{proof}[Proof of Theorem~\ref{theorem:main}] \emph{(i)} If $Q_{A,p,V}$ is critical in $\Om$.  We claim that for any $\emptyset\neq K\Subset\Om$
\be\label{c_K=0}
 c_K
  :=\inf_{\substack{0\leq u\in\test(\Om)\\ \|u\|_{L^p(K)}=1}}Q_{A,p,V}[u]
   =0.
\ee
 To see this, pick $W\in\test(K)\setminus\{0\}$ such that $0\leq W\leq1$. Then
 \be\nonumber
  c_K\int_\Om W|u|^p\dx
   \leq
    c_K
     \leq
      Q_{A,p,V}[u],\qquad \mbox{for all }u\in\test(\Om)\mbox{ with }\|u\|_{L^p(K)}=1,
 \ee
a contradiction to the criticality of $Q_{A,p,V}$ in case $c_K>0$. Picking one such $K$, (\ref{c_K=0}) implies the existence of a null sequence with respect to $Q_{A,p,V}$.

If $Q_{A,p,V}$ admits a null sequence, then by Theorem~\ref{crit}, equation (\ref{Q'=0}) admits a unique positive solution $v$, which is also its unique (regular if $p<2$) positive supersolution. Suppose now to the  contrary, that $Q_{A,p,V}$ is subcritical in $\Om$ with a nonzero nonnegative weight $W$. By the AP theorem we obtain a positive solution $w$ of the equation $Q_{A,p,V-W}'[u]=0$ which is readily another positive supersolution of (\ref{Q'=0}). This contradicts the uniqueness of $v$, and thus $Q_{A,p,V}$ has to be critical in $\Om$.

\medskip

\emph{(ii)} The sufficiency is captured by Theorem~\ref{crit}. To prove the necessity, let $v$ be the unique positive (super)solution of $Q'_{A,p,V}$ in $\Om$. By part \emph{(i)} we have that the nonexistence of null sequences with respect to $Q_{A,p,V}$ implies that $Q_{A,p,V}$ is subcritical in $\Om$. Now the same argument as in the proof of the necessity of the first statement of part \emph{(i)} implies that $v$ is not unique, a contradiction.

\medskip

\emph{(iii)} The necessity follows by the definition of subcriticality. On the other hand, the proof of the sufficiency of the first statement of part \emph{(i)} implies that $c_K>0$ for any domain $K\Subset\Om$. Using a standard partition of unity argument we arrive at a strictly positive $W$ that satisfies (\ref{ineq:subcriticality}) (see, \cite[Lemma~3.1]{PT1}).

\medskip

\emph{(iv)} The proof is identical to \cite[Theorem 1.6-(4)]{PT1} (and also \cite{PR}).
\end{proof}

\begin{corollary}\label{corollary:convex:combination}
Suppose that for $i=0,1$, the functional $Q_{A,p,V_i}$ is nonnegative in $\Om$ with $A,~V_i$ satisfying hypothesis {\em(H0)} if $p\geq2,$ or {\em (H1)} if $1<p<2$. For $t\in(0,1)$ set
$$V_t:=(1-t)V_0+tV_1.$$
Then $Q_{A,p,V_t}$ is nonnegative in $\Om$ for all $t\in[0,1]$. Moreover, if $\Ln\big(\{V_0\neq V_1\}\big)>0$, then $Q_{A,p,V_t}$ is subcritical in $\Om$ for any $t\in(0,1)$.
\end{corollary}

\begin{proof} The nonnegativity of $Q_{A,p,V_t}$ for $t\in(0,1)$ follows from the obvious relation
\be\label{eq:Vt}
 Q_{A,p,V_t}[u]
  =
   (1-t)Q_{A,p,V_0}[u]+tQ_{A,p,V_1}[u].
\ee
Suppose now that $\{u_k\}\subset\test(\Om)$ is a null sequence with respect to $Q_{A,p,V_t}$ in $\Om$ for some $t\in(0,1)$, such that $u_k\rightarrow\phi$ in $L^p_{\rm loc}(\Om)$. It follows from (\ref{eq:Vt}) that $\{u_k\}$ is also a null sequence for $Q_{A,p,V_0}$ and $Q_{A,p,V_1}$ in $\Om$. By Theorem~\ref{crit}, $\phi$ is a solution of $Q'_{A,p,V_i}[u]=0$ in $\Om$, for both values of $i$, which is impossible since $\Ln\big(\{V_0\neq V_1\}\big)>0$.
\end{proof}

Finally, we state generalizations of the corresponding results in \cite{PT1,PR}. We skip their proofs since they are essentially the same.

\begin{proposition}\label{proposition:superdomain}
Suppose $\Om'\subsetneq\Om$ is a domain. Let $A$ and $V$  satisfy hypothesis {\em(H0)} in case $p\geq2,$ or {\em (H1)} if $1<p<2$.

a) If $Q_{A,p,V}$ is nonnegative in $\Om$, then $Q_{A,p,V}$ is subcritical in $\Om'$.

b) If $Q_{A,p,V}$ is critical in $\Om'$, then $Q_{A,p,V}$ is supercritical in $\Om$.
\end{proposition}

\begin{proposition}\label{proposition:used in loc unif convergence}
Suppose that $Q_{A,p,V}$ is subcritical in $\Om$ with $A$ and $V$  satisfying hypothesis {\em(H0)} if $p\geq2,$ or {\em (H1)} if $1<p<2$. Let $U\in L^\infty(\Om)\setminus\{0\}$ such that $U\geq0$ and $\sprt\{U\}\Subset\Om$. Then there exist $\tau_+>0$ and $\tau_-\in[-\infty,0)$ such that $Q_{A,p,V+tU}$ is subcritical in $\Om$ if and only if $t\in(\tau_-,\tau_+)$ and $Q_{A,p,V+\tau_+U}$ is critical in $\Om$.
\end{proposition}

\begin{proposition}
Suppose that $Q_{A,p,V}$ is critical in $\Om$ with $A$ and $V$  satisfying hypothesis {\em(H0)} if $p\geq2,$ or {\em (H1)} if $1<p<2$. Denote by $\phi$ the corresponding ground state. Consider $U\in L^\infty(\Om)$ such that $\sprt\{U\}\Subset\Om$. Then there exists $0<\tau_+\leq\infty$ such that $Q_{A,p,V+tU}$ is subcritical in $\Om$ for $t\in(0,\tau_+)$ if and only if $\int_{\Om}U|\phi|^{p}\dx>0$.
\end{proposition}
The following theorem extends the corresponding theorems in \cite{P,PR,PTT}; see some applications therein.
\begin{theorem}\label{theorem:Liouville}{\bf [Liouville comparison theorem]} 
Suppose that for $i=1,2,$ the functional $Q_{A_i,p,V_i}$ is nonnegative in $\Om$ with $A_i,~V_i$ satisfying hypothesis {\em(H0)} if $p\geq2,$ or {\em (H1)} if $1<p<2$. Suppose in addition that:
\begin{itemize}
 \item[(i)] $Q_{A_2,p,V_2}$ admits a ground state $\phi$ in $\Om$.
 \item[(ii)] The equation $Q'_{A_1,p,V_1}[u]=0$ in $\Om$ admits a weak subsolution $\psi$ with $\psi^+\neq0$.
 \item[(iii)] There exists $M>0$ such that the matrix $\big(M\phi(x)\big)^2A_1(x)-\big(\psi_+(x)\big)^{2}A_0(x)$ is nonnegative-definite in $\Rn$ for almost every $x\in\Om$.
 \item[(iv)] There exists $N>0$ such that $|\nabla \psi|^{p-2}_{A_{0}(x)} \leq N^{p-2}|\nabla \phi|^{p-2}_{A_{1}(x)}$ for almost every $x$ in $\Om\cap\{\psi>0\}$.
\end{itemize}
Then the functional $Q_{A_1,p,V_1}$ is critical in $\Om$, and $\psi$ is the unique positive supersolution of $Q'_{A_1,p,V_1}[u]=0$ in $\Om$.
\end{theorem}

We close this section by showing that the ground state is a locally-uniform limit of a null sequence. This is a generalization of the second statement of \cite[Theorem 6.1 (2)]{PR}. We give a detailed proof, as it utilizes many of the results presented above.

\begin{lemma}\label{lemma:unif:conv:gs}
Suppose $Q_{A,p,V}$ is critical in $\Om$ with $A$ and $V$  satisfying hypothesis {\em(H0)} if $p\geq2,$ or {\em (H1)} if $1<p<2$. Then $Q_{A,p,V}$ admits a null sequence that converges locally uniformly to the ground state.
\end{lemma}

\begin{proof} Let $\{\om_i\}_{i\in\N}$ be a sequence of Lipschitz domains such that $\om_i\Subset\Om$, $\om_i\Subset\om_{i+1}$ for $i\in\N$, and $\cup_{i\in\N}\om_i=\Om$. We fix $x_0\in\om_1$ and a nonnegative $U\in\test(\Om)\setminus\{0\}$ with $\sprt\{U\}\subset\om_1$. By Proposition~\ref{proposition:used in loc unif convergence}, for every $i\in\N$ there exists $t_i>0$, such that the functional $Q_{A,p,V-t_iU}$ is critical in $\om_i$. For $i\in\N$ we denote by $\phi_i\in W^{1,p}(\om_i)$ the corresponding ground states, normalized by $\phi_i(x_0)=1$. The sequence of $t_i$'s is strictly decreasing with $i$. Indeed, we have by Proposition~\ref{proposition:superdomain} that $Q_{A,p,V-t_iU}$ has to be supercritical in $\om_{i+1}$. There exists thus $u\in\test(\om_{i+1})$ such that $Q_{A,p,V-t_iU}[u;\om_{i+1}]<0$. This in turn implies that
 \be\nonumber
  Q_{A,p,V-t_{i+1}U}[u;\om_{i+1}]
   <
    (t_i-t_{i+1})\int_{\om_{i+1}}U|u|^p\dx.
 \ee
The criticality of $Q_{A,p,V-t_{i+1}U}$ in $\om_{i+1}$ implies by definition that $Q_{A,p,V-t_{i+1}U}$ is nonnegative in $\om_{i+1}$ and thus $t_i>t_{i+1}$. Setting $t_\infty:=\lim_{i\to\infty}t_i$, by Harnack's convergence principle (Proposition~\ref{Harnack:cp}), up to a subsequence, $\{\phi_i\}_{i\in\N}$ converges locally uniformly to a positive solution $v$ of the equation $Q'_{A,p,V-t_\infty U}[u]=0$ in $\Om$. The AP theorem (Theorem~\ref{theorem:AP}) implies that $Q_{A,p,V-t_\infty U}$ is nonnegative in $\Om$. Clearly, $t_\infty\geq0$. Let us show that in fact $t_\infty=0$. If not then $V-t_\infty U\leq V$ a.e. in $\Om$, and since by our assumptions $Q_{A,p,V}$ is critical in $\Om$, part $b)$ of Proposition~\ref{proposition:superpotential} gives that $Q_{A,p,V-t_\infty U}$ is supercritical, contradicting its nonnegativity.

Summarizing, for each $i\in\N$ we have obtained a ground state $\phi_i\in W^{1,p}(\om_i)$ of $Q_{A,p,V-t_iU}$ in $\om_i$, and the sequence $\{\phi_i\}_{i\in\N}$ converges locally uniformly to a positive solution $v$ of the equation (\ref{Q'=0}) in $\Om$. To conclude we will show that $\{\phi_i\}_{i\in\N}$ is in fact a null sequence. Consider the principal eigenvalue $\lambda_1(Q_{A,p,V-t_iU_i};\om_i);~i\in\N$, which is nonnegative. Suppose that for some $i\in\N$ we had $\lambda_1(Q_{A,p,V-t_iU_i};\om_i)>0$. Then the principal eigenfunction $v^{\om_i}_1\in W_0^{1,p}(\om_i)$ would be a positive, strict supersolution of the equation $Q'_{A,p,V-t_iU}[v;\om_i]=0$, which contradicts the fact that $\phi_i$ is the unique positive supersolution and also a solution of $Q'_{A,p,V-t_iU}[v;\om_i]=0$ (see Theorem~\ref{crit}). Thus $\lambda_1(Q_{A,p,V-t_iU_i};\om_i)=0$ for each $i\in\N$, and since $\phi_i$ is also the unique positive solution of $Q'_{A,p,V-t_iU}[v;\om_i]=0$ (see again Theorem~\ref{crit}) we conclude $\phi_i=v^{\om_i}_1\in W_0^{1,p}(\om_i)$. Consequently,
 \be\nonumber
  \lim_{i\to\infty}Q_{A,p,V}[\phi_i]
   =
    \lim_{i\to\infty}t_i\int_{\Om_1}U\phi_i^p\dx
     =0.
 \ee
After a further normalization, we may assume that for some $\emptyset\neq K\Subset\Om$, there also holds $\|\phi_i\|_{L^p(K)}=1$ for all $i\in\N$.
\end{proof}
\section{Positive solutions of minimal growth at infinity}\label{minimal:growth}
The present section is devoted to the existence of positive solutions of the equation $Q'_{A,p,V}[v]=0$ in $\Om\setminus\{x_0\}$ that have minimal growth at infinity in $\Om$, and their role in criticality theory. For this purpose we extend in the following subsection the \textit{weak comparison principle} (WCP) (cf. \cite{GMdL,PR}).
Subsection~\ref{ssubsec:isolated} is devoted to the study of the behaviour of positive solutions near an isolated singularity. Finally, in \S\ref{ssubsec:minimal} we study positive solutions of minimal growth at infinity in $\Om$, and prove the last two parts of the Main Theorem.
\subsection{Weak comparison principle (WCP)}\label{ssubsec:prelims:minimal}
We prove first a simple version of the WCP that holds true for the $p$-Laplacian operator with a {\em nonnegative} potential (see for instance \cite[Theorem 2.4.1]{PS}).
\begin{lemma}\label{lemma:wcp+}
Let $\om$ be a Lipschitz domain in $\Rn$. Suppose that $A$ is a uniformly elliptic and bounded matrix in $\om$, and $\G,~\mathcal{V}\in M^q(p;\om)$ with $\mathcal{V}\geq0$ a.e. in $\Om$. Suppose that $v_1$ (respectively, $v_2$) is a subsolution (respectively, supersolution) of the equation
 \be\label{Q'=f:positive}
  Q'_{A,p,\mathcal{V}}[v]=\G\hspace{1em}\mbox{in }\om.
 \ee
If $v_1\leq v_2$ a.e. on $\partial\om$ in the trace sense, then $v_1\leq v_2$ a.e. in $\om$.
\end{lemma}

\begin{proof}
 Our assumption that $v_1\leq v_2$ a.e. on $\partial\om$, implies $(v_2-v_1)^-\in W_0^{1,p}(\om)$. Using this as a test function in the definitions of $v_1,~v_2$ being respectively sub/supersolutions of (\ref{Q'=f:positive}), and subtracting the two resulting inequalities we obtain
 \bea\nonumber
  &  &\hspace{-2em}\int_\om\big(|\nabla v_1|^{p-2}_AA\nabla v_1-|\nabla v_2|^{p-2}_AA\nabla v_2\big)\cdot\nabla(v_2-v_1)^-\dx
   \\ \nonumber &  &\hspace{6em}+
    \int_\om\mathcal{V}\big(|v_1|^{p-2}v_1-|v_2|^{p-2}v_2\big)(v_2-v_1)^-\dx
     \leq 0.
 \eea
In other words
 \bea\nonumber
  &  &\hspace{-2em}\int_{\{v_2<v_1\}}\Big(\big(|\nabla v_1|^{p-2}_AA\nabla v_1-|\nabla v_2|^{p-2}_AA\nabla v_2\big)\cdot\big(\nabla v_1-\nabla v_2\big)\dx
   \\ \nonumber &  &\hspace{6em}+
    \mathcal{V}\big(|v_1|^{p-2}v_1-|v_2|^{p-2}v_2\big)(v_1-v_2)\Big)\dx
     \leq 0.
 \eea
By (\ref{cauchy-schwartz}) we have that each term of the sum of the integrand is nonnegative with equality if and only if $\nabla v_1=\nabla v_2$ a.e. in the set $\{v_2<v_1\}$, or what is the same $(v_2-v_1)^-=c\geq0$ a.e. in $\om$. Since $(v_2-v_1)^-=0$ a.e. on $\partial\om$ in the trace sense, we conclude $v_1\leq v_2$ a.e. in $\om$.
\end{proof}
The following proposition deals with the sub/supersolution technique.
\begin{proposition}\label{Diaz}
Let $\om$ be a Lipschitz domain in $\Rn$. Assume that $A$ is a uniformly elliptic and bounded matrix in $\om$, and $g,~V\in M^q(p;\om)$, where $g\geq 0$ a.e. in $\om$. Let $f,~\varphi,~\psi\in W^{1,p}(\om)\cap C(\bar{\om})$, where $f\geq 0$ a.e. in $\om$, and
\[\left\{
 \begin{array}{lll}
  Q'_{A,p,V}[\psi]\leq g\leq Q'_{A,p,V}[\varphi] & \mbox{in $\om$, in the weak sense}
   \\[2pt]
  \psi\leq f\leq\varphi & \mbox{on }\partial\om,
   \\[2pt]
  0 \leq \psi\leq\varphi & \mbox{in }\om.
 \end{array}
\right.
\]
Then there exists a nonnegative solution $u\in W^{1,p}(\om)\cap C(\bar{\om})$ of
\be\label{Q'=f}
 \left\{
 \begin{array}{ll}
   Q'_{A,p,V}[u]=g & \mbox{in }\om,
    \\[8pt]
   u=f & \mbox{on }\partial\om,
 \end{array}
\right.
\ee
such that $\psi\leq u\leq\varphi$ in $\om$.

Moreover, if $f>0$ a.e. in $\partial\om$, then the solution $u$ is the unique solution of \eqref{Q'=f}.
\end{proposition}

\begin{proof}
Consider the set
 \be\nonumber
  \mathcal{K}:=\big\{v\in W^{1,p}(\om)\cap C(\bar{\om})\;\mid \; 0\leq \psi\leq v\leq\varphi\mbox{ in }\om\big\}.
 \ee
For any $x\in\om$ and $v\in\mathcal{K}$ we define
 \be\nonumber
  G(x,v):=g(x)+2V^-(x)\big(v(x)\big)^{p-1}.
 \ee
Note that $G\in M^q(p;\om)$ and $G\geq0$ a.e. in $\om$. The map $T:\mathcal{K}\rightarrow W^{1,p}(\om)$ defined by $T(v)=u$, where $u$ is the solution of
 \be\label{Q'=g}\Bigg\{
  \begin{array}{ll}
   Q'_{A,p,|V|}[u]=G(x,v) &\mbox{in }\om,
    \\[8pt]
   u=f &\mbox{in the trace sense on }\partial\om,
  \end{array}
 \ee
is well defined by Propositions \ref{proposition:wlsc} and \ref{proposition:coercivity}. Indeed, consider the functionals
$$J,\; \bar{J}:W^{1,p}(\om)\to\R\cup\{\infty\}$$
defined respectively in (\ref{funct:barJ}) and (\ref{funct:J}), with $\mathcal{V}=|V|$ and $\G=G(x,v)$. Let
\[
 \{u_k\}_{k\in\N}\subset\mathcal{A}:=\{u\in W^{1,p}(\om)\mid u=f\mbox{ on }\partial\om\},
\]
be such that
\[
J[u_k]\downarrow m:=\inf_{u\in\mathcal{A}}J[u].
\]
Since $f\geq0$, we have that $\{|u_k|\}_{k\in\N}\subset\mathcal{A}$ as well, which implies $m\leq J[|u_k|]=\bar{J}[u_k]\leq J[u_k]$, the latter inequality holds since $\G\geq0$ a.e. in $\om$. In particular, it follows that $\inf_{u\in\mathcal{A}}\bar{J}[u]=m$. Letting $k\to\infty$ we deduce
\[
 \bar{J}[u_k]\rightarrow m.
\]
But, by Proposition \ref{proposition:wlsc}-($b$), $\bar{J}$ is weakly lower semicontinuous, and by Proposition \ref{proposition:coercivity}-($a$) it is also coercive. Since $\mathcal{A}$ is weakly closed, it follows (see for example, \cite[Theorem 1.2]{Str}) that  $m$ is achieved by a nonnegative function $u\in\mathcal{A}$ that satisfies $\bar{J}(u)=m$. Moreover, $J(u)=\bar{J}(u)=m$. So, $u$ is a minimizer of $J$ on $\mathcal{A}$, and hence a solution of \eqref{Q'=g}. 

\medskip

Observe that the map $T$ is monotone. Indeed, let $v_1,~v_2\in\mathcal{K}$ be such that $v_1\leq v_2$. Then since $G(x,v)$ is increasing in $v$ we have
 \[
  Q'_{A,p,|V|}[T(v_1);\om]=g(x,v_1)\leq g(x,v_2)=Q'_{A,p,|V|}[T(v_2);\om],
 \]
and since $T(v_1)=f=T(v_2)$ on $\partial\om$, we get from  Lemma \ref{lemma:wcp+} with $\mathcal{V}=|V|$ and $\G=g(x,v_1)$ that $T(v_1)\leq T(v_2)$ in $\om$.

\medskip

Let $v\in W^{1,p}(\om)\cap C(\bar{\om})$ be a subsolution of (\ref{Q'=f}). Then $Q'_{A,p,|V|}[v]=Q'_{A,p,V}[v]+G(x,v)-g(x)\leq G(x,v)$ in $\om$, in the weak sense, and thus $v$ is a subsolution of (\ref{Q'=g}). On the other hand, $T(v)$ is a solution of (\ref{Q'=g}). Lemma \ref{lemma:wcp+} with $\mathcal{V}=|V|$ and $\G=G(x,v)$ gives $v\leq T(v)$ a.e. in $\om$. This implies in turn that
 \begin{equation*}
  Q'_{A,p,V}[T(v)]
    =
    g+2V^-\big(|v|^{p-2}v-|T(v)|^{p-2}T(v)\big)
      \leq
      g \qquad \mbox{in }\om,
 \end{equation*}
in the weak sense.

Summarizing, if $v$ is a subsolution of (\ref{Q'=f}) then $T(v)$ is a subsolution of (\ref{Q'=f}) such that $v\leq T(v)$ a.e. in $\om$. In the same fashion, we can show that if $v\in W^{1,p}(\om)\cap C(\bar{\om})$ is a supersolution of (\ref{Q'=f}) then $T(v)$ is a supersolution of (\ref{Q'=f}) such that $v\geq T(v)$ a.e. in $\om$.

\medskip

Defining the sequences
\[
 \underline{u}_0:=\psi,\;\; \underline{u}_n:=T(\underline{u}_{n-1})=T^{(n)}(\psi), \;\; \mbox{and }\;\;
\overline{u}_0:=\varphi,\;\; \overline{u}_n:=T(\overline{u}_{n-1})=T^{(n)}(\varphi) \quad n\in\N,
\]
we get from the above considerations that  $\{\underline{u}_n\}$ and $\{\overline{u}_n\}$ increases and decreases, respectively, to functions $\underline{u}$ and $\overline{u}$ for every $x\in\om$. Moreover, the convergence is clearly also in $L^p(\om)$ (by Theorem~1.9 in \cite{LL}). Then, using an argument similar to the proof of Proposition~\ref{Harnack:cp}, it follows that $\underline{u}$ and $\overline{u}$ are fixed points of $T$, and both solve (\ref{Q'=f}) and satisfy $\psi\leq \underline{u}\leq \overline{u}\leq\phi$ in $\om$.

The uniqueness claim follows from part (iii) of Lemma~\ref{lemma:lind}.
\end{proof}

Finally, we extend the WCP (cf. \cite{GMdL,PR,PS})

\begin{theorem}[Weak comparison principle]\label{WCP}
Let $\om\subset\Rn$ be a bounded Lipschitz domain. Suppose that $A$ is a uniformly elliptic and bounded matrix in $\om$, and $g,~V\in M^q(p;\om)$ with $g\geq0$ a.e. in $\om$. Assume that $\lambda_1>0$, where $\lambda_1$ is the principal eigenvalue of the operator $Q'_{A,p,V}$ defined by (\ref{first:eigenvalue}). Let $u_2\in W^{1,p}(\om)\cap C(\bar{\om})$ be a solution of
\[
 \left\{
 \begin{array}{ll}
   Q'_{A,p,V}[u_2]=g & \mbox{ in }\om,
    \\[8pt]
   u_2>0 & \mbox{ on }\partial\om.
 \end{array}
 \right.
\]
If $u_1\in W^{1,p}(\om)\cap C(\bar{\om})$ satisfies
\[
 \left\{
 \begin{array}{ll}
   Q'_{A,p,V}[u_1]\leq Q'_{A,p,V}[u_2] & \mbox{ in }\om,
    \\[8pt]
   u_1\leq u_2 & \mbox{ on }\partial\om,
 \end{array}
 \right.
\]
then, $u_{1}\leq u_{2}$ in $\om$.
\end{theorem}

\begin{proof}
Since $u_2$ is a supersolution of (\ref{Q'=0}) in $\om$ that is positive on $\partial\om$, the strong maximum principle implies $u_2>0$ in $\bar{\om}$. Let $c:=\max\{1,\max_{\bar{\om}}u_1/\min_{\bar{\om}}u_2\}$, then $u_1\leq cu_2$ in $\bar{\om}$. Consider now the problem
\be\label{Dirichlet:Diaz}
 \Bigg\{
 \begin{array}{ll}
   Q'_{A,p,V}[v]=g & \mbox{in }\om,
    \\[8pt]
   v=u_2 & \mbox{on }\partial\om.
 \end{array}
\ee
By the choice of $c$ and our assumption we have that $cu_2$ is a supersolution of (\ref{Dirichlet:Diaz}) such that $u_1\leq u_2\leq cu_2$ on $\partial\om$, while $u_1$ is a subsolution of (\ref{Dirichlet:Diaz}). Applying Proposition \ref{Diaz} with $\psi=u_1$ and $\phi= cu_2$, we get a unique solution $v$ of (\ref{Dirichlet:Diaz}) such that $u_1\leq v\leq cu_2$ in $\om$ and $v=u_2$ on $\partial\om$, in the trace sense. Clearly, $v$ is a supersolution of (\ref{Q'=0}) in $\om$ that is positive on $\partial\om$. Again, by the strong maximum principle, we get $v>0$ in $\bar{\om}$. By the uniqueness of the boundary problem \eqref{Dirichlet:Diaz} (Proposition \ref{Diaz}), we have  $v=u_2$. Hence, $u_1\leq u_2$ in $\om$. 
\end{proof}
\subsection{Behaviour of positive solutions near an isolated singularity}\label{ssubsec:isolated}
Using the weak comparison principle of the previous subsection (Theorem~\ref{WCP}) we study the behaviour of positive solutions near an isolated singular point. We have

\begin{theorem}\label{theorem:removesing}
Let $p\leq n$ and $x_0\in\Om$. Suppose $A$ and $V$  satisfy hypothesis {\emph(H0)} in $\Om$, and let $u$ be a nonnegative solution of the equation $Q'_{A,p,V}[v]=0$ in $\Om\setminus\{x_0\}$.

1. If  $u$ is bounded near $x_0$, then $u$ can be extended to a positive solution in $\Om$.

2. If $u$ is unbounded near $x_0$, then $\displaystyle{\lim_{x\to x_0}}u(x)=\infty$.
\end{theorem}

\begin{proof}1. This is a special case of \cite[Theorem 3.16]{MZ}, which is in turn an extension to $V\in M_{\rm loc}^q(p;\Om)$ of \cite[Theorem 10]{Sr}, where $V$ is assumed to be in $L^q_{\rm loc}(\Om)$ for some $q>n/p$. In particular, this part of the theorem holds true for solutions of arbitrary sign in $\Om\setminus o$, where $o$ is a set having zero $p$-\textit{capacity}.

2. We follow the argument in  \cite{FP1} (for a bit different argument see \cite[p. 278]{Sr}). Without loss of generality, we assume that $x_0=0$ and $B_1(0)\Subset \Om$. For $r>0$, we denote the ball $B_{r}:=B_{r}(0)$, and the corresponding sphere $S_{r}:=\partial B_{r}$.

Since $\limsup_{x\to 0}u(x)=\infty,$ there exists a sequence $\{x_k\}_{k\in\N}\subset\Om$ converging to $0,$ such that $u(x_k)\to\infty$ as $k\to\infty$. Let $r_k=|x_k|$, where $k=1,2,\ldots$, and consider the annular domains $\mathbb{A}_{k}:=B_{3r_k/2}\setminus\bar{B}_{r_k/2}$. For each $k$ we scale $\mathbb{A}_{k}$ to the fixed annulus $\mathbb{A}':=B_{3/2}(0)\setminus\bar{B}_{1/2}(0)$. Note next that if $u$ is a solution of the equation  $Q'_{A,p,V}[v]=0$ in $\Om\setminus\{0\}$, then for any positive $R$, the function $u_R(x):=u(Rx)$  satisfies the equation
\begin{equation}\label{QARVR}
     Q'_{A_R,p,V_R}[u_R]:=-\Div_{A_R}\big\{|\nabla u_R|^{p-2}_{A_R}A_R(x)\nabla u_R\big\} + V_R(x)|u_R|^{p-2}u_R=0\quad\mbox{in }\Om_R,
\end{equation}
where $A_R(x):=A(Rx)$, $V_R(x):=R^pV(Rx)$, and $\Om_R:=\{x/R~|~x\in\Om\setminus \{0\}\}$. Applying thus the Harnack inequality in $\mathbb{A}'$, we have for $k$ sufficiently large
\be\label{harnack:annuli}
 \sup_{x\in\mathbb{A}_{k}}u(x)=\sup_{x\in\mathbb{A}'}u_{r_k}(x)\leq C \inf_{x\in\mathbb{A}'}u_{r_k}(x)=C\inf_{x\in\mathbb{A}_{k}}u(x),
\ee
where the positive constant $C$ is independent of $r_k$. To see this for example in the case $p<n,$ observe that $\|V_R\|_{M^q(\mathbb{A}')}=R^{p-n/q}\|V\|_{M^q(\mathbb{A}_{R})}$ and by our assumptions on $q$ we have that the exponent on $R$ is nonnegative (it is in fact positive). Now from (\ref{harnack:annuli}) we may readily deduce
\be\label{eq_subseq}
     \min_{S_{r_k}}u(x)\rightarrow\infty\qquad\mbox{as }k\to\infty.
\ee
Let $v$ be a fixed positive solution of the equation $Q'_{A,p,V}[w]=0$ in $B_1$, and set for $0<r<1$
\[
 m_r:=\min_{S_r} \frac{u(x)}{v(x)}\,.
\]
Then, as in \cite[Lemma~4.2]{FP1}, the WCP implies that the function $m_r$ is monotone as $r\to 0$. This together with \eqref{eq_subseq} imply that $m_r$ is monotone nondeceasing near $0$.  Therefore, $\lim_{r\to 0} m_r = \infty$, and thus, $\lim_{x\to 0} u(x) = \infty$.
\end{proof}

\begin{remark}
The asymptotic behavior of positive solutions of the equation $Q'_{A,p,V}[v]=0$ near an isolated singular point remains open for further studies (see \cite{FP1,FP2,PT2} and the references therein for partial results).
\end{remark}
\subsection{Positive solutions of minimal growth and Green's function}\label{ssubsec:minimal}
The following notion was introduced by Agmon \cite{Agm} in the linear case and was extended to $p$-Laplacian type equations of the form (\ref{intro:pLaplace}) in \cite{PT1} and \cite{PR}.
\begin{definition}
Let $K_0$ be a compact subset of $\Om$. A positive solution $u$ of (\ref{Q'=0}) in $\Om\setminus K_0$ is said to be a \emph{positive solution of minimal growth in a neighborhood of infinity in $\Om$,} and denoted by $u\in\mathcal{M}_{\Om;K_0}$, if for any smooth compact subset of $\Om$ with $K_0\Subset{\rm int}K$, and any positive supersolution $v\in C((\Om\setminus{\rm int}K)$ of (\ref{Q'=0}) in $\Om\setminus K$, we have
 \be\nonumber
  u\leq v\mbox{ on }\partial K\;\;\;\Rightarrow\;\;\;u\leq v\mbox{ in }\Om\setminus K.
 \ee
If $u\in\mathcal{M}_{\Om;\emptyset}$, then $u$ is called a \emph{global minimal solution of (\ref{Q'=0}) in $\Om$.}
\end{definition}
We first prove that if $Q_{A,p,V}$ is nonnegative in $\Om$, then for any $x_0\in\Om$, $\mathcal{M}_{\Om;\{x_0\}}\neq \emptyset$. This result extends the corresponding results in  \cite{PT1,PT2}, and \cite{PR}.

\begin{theorem}
Suppose that $Q_{A,p,V}$ is nonnegative in $\Om$ where $A$ and $V$  satisfy hypothesis {\em(H0)}. Then for any $x_0\in\Om$, the equation $Q'_{A,p,V}[v]=0$ admits a solution $u\in\mathcal{M}_{\Om;\{x_0\}}$.
\end{theorem}

\begin{proof} We fix a point $x_0\in\Om$ and let $\{\om_i\}_{i\in\N}$ be a sequence of Lipschitz domains such that $x_0\in\om_1$, $\om_i\Subset\om_{i+1}\Subset \Om$, where $i\in\N$, and $\cup_{i\in\N}\om_i=\Om$. Setting $r_1:=\sup_{x\in\om_1}\dist(x;\partial\om_1)$ (the inradius of $\om_1$), we define the open sets
 \be\nonumber
  U_{i}
   :=
    \om_i\setminus\overline{B}_{r_1/(i+1)}(x_0).
 \ee
Pick a fixed reference point $x_1\in U_1$ and note that $U_i\Subset U_{i+1};~i\in\N$, and also $\cup_{i\in\N}U_i=\Om\setminus\{x_0\}$.  Let also $f_i\in\test\big(B_{r_1/i}(x_0)\setminus\overline{B}_{r_1/(i+1)}(x_0)\big)\setminus\{0\}$ be a sequence of nonnegative functions. The nonnegativity of $Q_{A,p,V}$ implies $\lambda_1(Q_{A,p,V+1/i};~U_i)>0$, and thus by Theorem \ref{theorem:GL} we obtain for each $i\in\N$, a positive solution $v_i$ of
\[
 \Bigg\{
 \begin{array}{ll}
   Q'_{A,p,V+1/i}[v]=f_i & \mbox{in }U_i,
    \\[8pt]
   v=0 & \mbox{on }\partial U_i.
 \end{array}
\]
Normalizing by $u_i(x):=v_i(x)/v_i(x_1)$, the Harnack convergence principle (Proposition \ref{Harnack:cp}) implies that $\{u_i\}_{i\in\N}$ admits a subsequence converging uniformly in compact subsets of $\Om\setminus\{x_0\}$ to a positive solution $u$ of \eqref{Q'=0}.

We claim that $u\in\mathcal{M}_{\Om;\{x_0\}}$. To this end, let $K$ be a compact smooth subset of $\Om$ such that $x_0\in{\rm int}K$, and let $v\in C(\Om\setminus{\rm int}K)$ be a positive supersolution of \eqref{Q'=0} in $\Om\setminus K$ with $u\leq v$ on $\partial K$. Let $\delta>0$. There exists then $i_K\in\N$ such that $\sprt\{f_i\}\Subset K$ for all $i\geq i_K$,  and in addition $u_i\leq(1+\delta)v$ on $\partial(\om_i\setminus K)$. The WCP (Theorem \ref{WCP}) implies $u_i\leq(1+\delta)v$ in $\om_i\setminus K$, and letting $i\to\infty$ we obtain $u\leq(1+\delta)v$ in $\Om\setminus K$. Since $\delta>0$ is arbitrary we conclude $u\leq v$ in $\Om\setminus K$.
\end{proof}

\begin{definition}
A function $u\in \mathcal{M}_{\Omega,\left\{x_{0}\right\}}$ having a nonremovable singularity at $x_{0}$ is called a {\em minimal positive Green function of
$Q'_{A,V}$ in $\Om$ with a pole at $x_0$}. We denote such a function by $G_{A,V}^\Om (x,x_0)$.
\end{definition}
The following theorem states that criticality is equivalent to the existence of a global minimal solution, that is $\mathcal{A}_1\Leftrightarrow\mathcal{A}_5$ in the Main Theorem presented in the introduction. It extends \cite[Theorem 9.6]{PR} and also \cite[Theorem 5.5]{PT1} and \cite[Theorem 5.8]{PT2}.

\begin{theorem}\label{global thm}
Suppose that $Q_{A,p,V}$ is nonnegative in $\Om$ with $A$ and $V$  satisfying hypothesis {\em(H0)} if $p\geq2,$ or {\em (H1)} if $1<p<2$. Then $Q_{A,p,V}$ is subcritical in $\Om$ if and only if (\ref{Q'=0}) does not admit a global minimal solution in $\Om$. In particular, $\phi$ is a ground state of (\ref{Q'=0}) in $\Om$ if and only if $\phi$ is a global minimal solution of (\ref{Q'=0}) in $\Om$.
\end{theorem}

\begin{proof} To prove necessity, let $Q_{A,p,V}$ be subcritical in $\Om$. Clearly (By the AP theorem) there exists a continuous positive strict supersolution $v$ of (\ref{Q'=0}) in $\Om$. We proceed by contradiction. Suppose there exists a global minimal solution $u$ of (\ref{Q'=0}) in $\Om$ and fix $K$ to be a compact smooth subset of $\Om$. Let $\varepsilon_{\partial K}:=\min_{\partial K}v/\max_{\partial K}u$. Then $\varepsilon_{\partial K}u\leq v$, and $\varepsilon^{-1}_{\partial K}v$ is also a positive continuous supersolution of (\ref{Q'=0}) in $\Om$. Using it as a comparison function in the definition of $u\in\mathcal{M}_{\Om;\emptyset}$, we get $\varepsilon_{\partial K}u\leq v$ in $\Om\setminus K$. Letting also $\varepsilon_K:=\min_{K}v/\max_{K}u$, we readily have $\varepsilon_Ku\leq v$ in $K$. Consequently, by setting $\varepsilon:=\min\{\varepsilon_{\partial K},\varepsilon_K\}$ we have
 \be\nonumber
  \varepsilon u\leq v\hspace{1em}\mbox{in }\Om.
 \ee
Now we define
 \be\nonumber
  \varepsilon_0
   :=
    \max\{\varepsilon>0\mbox{ such that }\varepsilon u\leq v\mbox{ in }\Om\},
 \ee
and note that since $\varepsilon_0u$ and $v$ are respectively, a continuous solution and a continuous strict supersolution of (\ref{Q'=0}) in $\Om$, we have $\varepsilon_0u\not\equiv v$. There exist thus $x_1\in \Om$, and $\delta,r>0$ such that $B_r(x_1)\subset\Om$ and
 \be\nonumber
  (1+\delta)\varepsilon_0u(x)
   \leq
    v(x)
     \hspace{1em}\mbox{for all }x\in\overline{B}_r(x_1).
 \ee
But since $u\in\mathcal{M}_{\Om;\emptyset}$ it follows that
 \be\nonumber
  (1+\delta)\varepsilon_0u(x)
   \leq
    v(x)
    \hspace{1em}\mbox{for all }x\in\Om\setminus\overline{B}_r(x_1).
 \ee
Consequently, $(1+\delta)\varepsilon_0u(x)\leq v(x)$ in $\Om$, which contradicts the definition of $\varepsilon_0$. We note that in the proof of this part, we did not use the further regularity assumption (H1).

\medskip

\noindent To prove sufficiency, assume that $Q_{A,p,V}$ is critical in $\Om$ with ground state $\phi$ satisfying $\phi(x_1)=1$, for some $x_1\in\Om$. We will prove that $\phi\in\mathcal{M}_{\Om;\emptyset}$. To this end, consider an exhaustion $\{\om_i\}_{i\in\N}$ of $\Om$ such that $x_0\in\om_1$ and $x_1\in\Om\setminus\om_1$. Fix $j\in\N$, and let $f_j\in\test(B_{r_1/j}(x_0))\setminus\{0\}$ satisfy $0\leq f_j(x)\leq1$, where as in the previous proof we write $r_1$ for the inradius of $\om_1$. Let $v_{i,j}$ be a positive solution of
\[
 \left\{
 \begin{array}{ll}
   Q'_{A,p,V}[v]=f_j & \mbox{in }\om_i,
    \\[8pt]
   v=0 & \mbox{on }\partial\om_i.
 \end{array}
 \right.
\]
The WCP (Theorem \ref{WCP}) ensures that the sequence $\{v_{i,j}\}_{i\in\N}$ is nondecreasing. If $\{v_{i,j}(x_1)\}$ is bounded, then the sequence converges to $v_j$, where $v_j$ is such that $Q'_{A,p,V}[v_j]=f_j$ in $\Om$. Thus $v_j$ would be a strict supersolution of (\ref{Q'=0}), which contradicts Theorem \ref{theorem:main}, since the ground state $\phi$ is the only positive supersolution of $Q'_{A,p,V}[w]=0$ in $\Om$. Therefore, $v_{i,j}(x_1)\to\infty$ as $i\to\infty$. Defining thus the normalized sequence $u_{i,j}(x):=\frac{v_{i,j}(x)}{v_{i,j}(x_1)}$, by the Harnack convergence principle (Proposition \ref{Harnack:cp}) we may extract a subsequence of $\{u_{i,j}\}$ that converges as $i\to\infty$ to a positive solution $u_j$ of the equation (\ref{Q'=0}) in $\Om$. Once again by the uniqueness of the ground state, we have $u_j=\phi$.

Now let $K$ be a smooth compact set of $\Om$ and assume that $x_0\in{\rm int}(K)$. Let $v\in C(\Om\setminus{\rm int}K)$ be a positive supersolution of (\ref{Q'=0}) in $\Om\setminus K$ such that $\phi\leq v$ on $\partial K$. Let $j\in\N$ be large enough, so that $\sprt\{f_j\}\Subset K$. For any $\delta>0$ there exists $i_\delta\in\N$ such that for $i\geq i_\delta$ to have
\[
\begin{cases}
0=Q'_{A,p,V}[u_{i,j}]\leq Q'_{A,p,V}[v] &\mbox{in }\om_i\setminus K,\\
Q'_{A,p,V}[v]\geq 0 &\mbox{in }\om_i\setminus K,\\
0\leq u_{i,j}\leq (1+\delta)v &\mbox{on }\partial(\om_i\setminus K),\\
\end{cases}
\]
which implies that $\phi=u_j\leq(1+\delta)v$ in $\Om\setminus K$. Letting $\delta\to0$ we obtain $\phi\leq v$ in $\Om\setminus K$.
\end{proof}
To conclude the paper, it remains to establish the equivalence between $\mathcal{A}_1$ and $\mathcal{A}_6$ of the Main Theorem.
\begin{theorem} \label{thrm:nonremove}
Suppose that $Q_{A,p,V}$ is nonnegative in $\Om$ with $A$ and $V$  satisfying hypothesis {\em(H0)} if $p\geq2,$ or {\em (H1)} if $1<p<2$. Let $u\in\mathcal{M}_{\Om,\{x_0\}}$ for some $x_0\in\Om$.

a) If $u$ has a removable singularity at $x_0$, then $Q_{A,p,V}$ is critical in $\Om$.

b) Let $1<p\leq n$, and suppose that $u$ has a nonremovable singularity at $x_0$, then $Q_{A,p,V}$ is subcritical in $\Om$.

c) Let $p>n$, and suppose that $u$ has a nonremovable singularity at $x_0$. Assume further that $\lim_{x\to x_0}u(x)=c$, where $c$ is a positive constant.
Then $Q_{A,p,V}$ is subcritical in $\Om$.
\end{theorem}

\begin{proof} $a$) If $u$ has a removable singularity at $x_0$, its continuous extension is a global minimal solution in $\Om$, and Theorem \ref{global thm} assures that $Q_{A,p,V}$ is critical in $\Om$.

$b$) Assume that $u$ has a nonremovable singularity at $x_0$, and suppose for the sake of contradiction that $Q_{A,p,V}$ is critical in $\Om$. Theorem \ref{global thm} implies the existence of a global minimal solution $v$ of \eqref{Q'=0} in $\Om$. By Theorem \ref{theorem:removesing} we have $\lim_{x\to x_0}u(x)=\infty$, and thus by comparison $v\leq\varepsilon u$ in $\Om$, where $\e$ is an arbitrary positive constant. This implies that $v=0$, a contradiction.

$c$) Suppose that $Q_{A,p,V}$ is critical in $\Om$, and let $v>0$ be the corresponding global minimal solution. We may assume that $v(x_0)=c$. Since both $u$ and $v$ are continuous at $x_0$, it follows that for any $\varepsilon>0$ there exists $\delta_\varepsilon>0$ such that for all $0<\delta<\delta_\varepsilon$
$$(1-\varepsilon)u(x)\leq v(x) \leq (1+\varepsilon)u(x)\qquad \forall x\in \partial B_{\delta}(x_0).$$
Since $u$ and $v$ are positive solutions (in $\Om\setminus \{x_0\}$ and $\Om$, respectively)  of minimal growth at infinity in $\Om$, the above inequality implies that
$$(1-\varepsilon)u(x)\leq v(x) \leq (1+\varepsilon)u(x)\qquad \forall x\in \Om\setminus \{x_0\}.$$
Letting $\e\to 0$, we get $u=v$ in $\Omega$, which contradicts our assumption that $u$ has a nonremovable singularity at $x_0$.
\end{proof}
\begin{remark}
For sufficient conditions ensuring that in the subcritical case with $p>n$, the limit of the Green function $G_{A,V}^\Om (x,x_0)$ as $x\to x_0$  always exists and is positive, see \cite{FP2}.
\end{remark}
\medskip
\begin{center}
{\bf Acknowledgments}
\end{center}
We would like to thank Professor Jan Mal\'{y} for explaining that (\ref{ineq:uncertainty}) holds in extension domains. We also thank Mr. Amir Kassis for providing us with a simpler proof of \cite[Proposition 4.3]{PT1} which in turn we have followed to prove Corollary~\ref{corollary:convex:combination}. The authors acknowledge the support of the Israel Science Foundation (grants No. 970/15) founded by the Israel Academy of Sciences and Humanities. G.P. is supported in part at the Technion by a Fine Scholarship.

{\sc Yehuda Pinchover}

Technion - Israel Institute of Technology

Department of Mathematics

Haifa 32000, Israel

E-mail: \verb"pincho@techunix.technion.ac.il"

\medskip

\&

\medskip

{\sc Georgios Psaradakis}

Technion - Israel Institute of Technology

Department of Mathematics

Haifa 32000, Israel

E-mail: \verb"georgios@techunix.technion.ac.il"
\end{document}